\documentclass[12pt]{paper}

\usepackage{amssymb,amsfonts,amsthm,amsmath}

\usepackage{latexsym,multicol,graphicx}

\usepackage{graphicx, verbatim}

\usepackage{xcolor}
\usepackage{verbatim}

\usepackage{pict2e}

\usepackage{fancyhdr}
\newtheorem{theorem}{Theorem}
\newtheorem*{theorem1}{Proposition A}
\newtheorem{lemma}{Lemma}

\renewcommand{\epsilon}{\varepsilon}
\renewcommand{\phi}{\varphi}
\renewcommand{\emptyset}{\varnothing}

\newcommand{\im}{{\rm i}}

\newcommand{\e}{\varepsilon}

\newcommand{\la}{\lambda}

\renewcommand{\le}{\leqslant}
\renewcommand{\ge}{\geqslant}

\newcommand{\bR}{\mathbb R}
\newcommand{\bC}{\mathbb C}
\newcommand{\bZ}{\mathbb Z}
\newcommand{\bD}{\mathbb D}
\newcommand{\bT}{\mathbb T}
\newcommand{\bN}{\mathbb N}
\newcommand{\bQ}{\mathbb Q}

\newcommand{\bP}{\mathbb P}

\newcommand{\bE}{\mathbb E}

\newcommand{\rd}{{\rm d}}

\newcommand{\dist}{\operatorname{dist}}

\numberwithin{equation}{subsection}
\numberwithin{lemma}{subsection}

\usepackage[figurename=Figure]{caption}

\date{}

\begin{document}

\title{Entire functions of exponential type represented
by pseudo-random and random Taylor series}
\author{
Alexander Borichev \and Alon Nishry \and Mikhail Sodin}

\maketitle

\hfill{To Alex Eremenko on occasion of his birthday}

\begin{abstract}We\let\thefootnote\relax\footnote{A.~Nishry and M.~Sodin were supported by Grant No.~166/11 of the Israel Science Foundation of the
Israel Academy of Sciences and Humanities. A.~Nishry was supported by U.S. National Science
Foundation Grant DMS-1128155.
M.~Sodin was supported by Grant No.~2012037 of
the United States--Israel Binational Science Foundation.} study the influence of the multipliers $\xi (n)$ on the angular
distribution of zeroes of the Taylor series
\[
F_\xi (z) = \sum_{n\ge 0} \xi (n) \frac{z^n}{n!}\,.
\]
We show that the distribution of zeroes of $ F_\xi $ is
governed by certain autocorrelations of the sequence $ \xi $.
Using this guiding principle, we consider several examples
of random and pseudo-random sequences $\xi$ and, in particular,
answer some questions posed by Chen and Littlewood in 1967. 

As a by-product we show that if $\xi$ is a stationary random integer-valued
sequence, then either it is periodic, or its spectral measure has no
gaps in its support. The same conclusion is true if $\xi$ is a complex-valued
stationary ergodic sequence that takes values in a uniformly discrete set.
\end{abstract}

\section{Introduction}
In
 this work, we consider entire functions of exponential type represented by the
Taylor series
\[
F_\xi (z) = \sum_{n\ge 0} \xi(n) \frac{z^n}{n!}\,, \qquad \xi\colon \bZ_+ \to \bC\,.
\]
We are interested in the influence of the multipliers $\xi(n)$ on the angular
distribution of zeroes of the function $F_\xi$. This question belongs to a ``terra
incognita'' in the theory of entire functions that contains no general results going in
this direction but several interesting examples. These examples include:

\medskip\noindent
(a) random independent identically distributed $\xi (n)$ (Littlewood--Offord~\cite{LO}, Kabluchko--Zaporozhets~\cite{KZ}),

\medskip\noindent
(b) $\xi (n) = e(q n^2) $ with
quadratic irrationality $q$ (Nassif~\cite{Nassif},
Littlewood~\cite{Littlewood}) and, more generally, arbitrary irrational $q$ (Eremenko--Ostrovskii~\cite{EO}),

\medskip\noindent
(c) $\xi (n) = e(n(\log n)^\beta)$ with $\beta>1$, and $e(n^\beta)$ with $1<\beta<\frac32$ (Chen--Littlewood~\cite{CL}),

\medskip\noindent
(d) uniformly almost periodic $\xi (n)$ (Levin\cite[Chapter~VI, \S 7]{Levin}),

\medskip\noindent Here and elsewhere, $e(t)=e^{2\pi\rm t\im}$.

\medskip
In this work, we consider the following four sequences $\xi$:

\medskip\noindent (i)
$\xi (n) = e(Q(n))$, where $Q(x)=\sum_{k\ge 2}q_kx^k$ is a polynomial  
with real coefficients $q_k$, at least one of which is irrational.

\medskip\noindent (ii)
$\xi (n) = e(n^\beta)$, where $\beta\ge \frac32$ is non-integer.

\medskip\noindent (iii)
$\xi (n)$ is a stationary sequence with a mild decay of the
maximal correlation coefficient.

\medskip\noindent (iv)
$\xi (n)$ is a stationary Gaussian sequence.

\medskip In the cases
(i), (ii), and (iii),
using some potential theory, we reduce the question on
the asymptotic distribution of zeroes of $F_\xi$ to certain lower bounds for the exponential sums
\[
W_R(\theta) = \sum_{|n|\le N} \xi (n+R) e(n\theta) e^{-\frac{n^2}{2R}}
\]
when $R \gg 1$ and $N$ has the size $R^{\frac12 + \epsilon}$ 
(see Lemmas~\ref{lemma4} and~\ref{lemma4.5}).
These lower bounds, in turn, depend on the behaviour of the autocorrelations 
\[
m \mapsto \frac1{N}\, \sum_{n=1}^N \xi (n+R)\,\overline{\xi (n+m+R)}\, e(m\theta)\,.
\]
In the case (iv) (similarly to the almost-periodic case (d)),
the zero set of $F_\xi$ has an angular density that, generally speaking, is
not constant, as in the cases (i), (ii) and (iii). This density is determined by the spectrum
of the sequence $\xi$, that is, after all, also by the autocorrelations between the
elements of $\xi$.

\subsubsection*{Acknowledgments}
We thank Alex Eremenko, Fedja Nazarov, and Benjy Weiss for illuminating  discussions.
We are grateful to the referee for helpful comments and remarks.

\section{Main results}

\subsection{}
We start with the cases when the zeroes of $F_\xi$ have the uniform angular distribution.

\medskip\noindent{\bf Definition.}
We say that the sequence $\xi\colon \bZ_+\to\bC$ is {\em an $L$-sequence},
if
\begin{equation}\label{eq:log-as}
\frac{\log |F_\xi (tz)|}{t} \underset{t\to\infty}\longrightarrow |z|\,, \qquad {\rm in\ }L^1_{\rm loc}(\mathbb C)\,.
\end{equation}

\medskip
Since $L^1_{\rm loc}(\mathbb C)$ convergence implies convergence in the sense of distributions, and the Laplacian is continuous in the distributional topology, ~\eqref{eq:log-as} yields
\begin{equation}\label{eq:zeroes-as}
\frac1{t}\, \Delta \log |F_\xi (tz)| \underset{t\to\infty}\longrightarrow \Delta |z| = \rd r \otimes \rd \theta\,,
\qquad z=re^{\im\theta}\,,
\end{equation}
in the sense of distributions, with $rd r \otimes \rd \theta$ being planar Lebesgue measure. 
Denoting by $n_F(r; \theta_1, \theta_2)$ the number of zeroes (counted with multiplicities)
of the entire function $F$ in the sector $\bigl\{z\colon 0\le |z|\le r, \theta_1 \le \arg (z) < \theta_2 \bigr\}$
and recalling that $\frac1{2\pi}\, \Delta \log |F|$ is the sum of point masses at zeroes of $F$, we
can rewrite~\eqref{eq:zeroes-as} in a more traditional form:
for every $\theta_1<\theta_2$,
\begin{equation}\label{eq:zeroes-as'}
n_{F_\xi}(r; \theta_1, \theta_2) = \frac{(\theta_2 - \theta_1 + o(1))\,r}{2\pi}
\qquad {\rm as\ } r\to\infty\,.
\end{equation}

\begin{theorem}\label{thm:Q}
Suppose that
$$ 
Q(x)=\sum_{k=2}^d q_k x^k
$$
is a polynomial with real coefficients $q_k$
and that at least one of the coefficients $q_k$ is irrational. Then $\xi (n) = e(Q(n))$ is an $L$-sequence.
\end{theorem}

For $Q(x)=qx^2$, $q$ being a quadratic irrationality, this is a result of Nassif~\cite{Nassif}
and Littlewood~\cite{Littlewood}. For arbitrary irrational $q$'s, this was proven by Eremenko
and Ostrovskii~\cite{EO}. It seems that the methods used in these works cannot be extended
to polynomials $Q$ of degree bigger than $2$.
Quoting Chen and Littlewood~\cite{CL}, ``many lines of experience converge to show that there
can be nothing doing if $\Lambda (n) \succ n^2$\,'' (in their notation, $\xi (n)=e(\Lambda (n))$, and
$\Lambda (n) \succ n^2$ means that $\Lambda (n)/n^2 \to \infty$)

\begin{theorem}\label{thm:beta}
For any non-integer $\beta>1$, the sequence $\xi (n) = e(n^\beta)$ is an
$L$-sequence.
\end{theorem}

As we have already mentioned, the case $1<\beta<\frac32$ is due to Chen and Littlewood~\cite{CL}.
They used the Poisson summation combined with the saddle point
approximation and obtained much more accurate information about the asymptotic location of
zeroes of the function $F_\xi$. They write: ``The gap $\frac32 \le\beta<2$ presents a most interesting
unsolved problem''.

\subsection{}
Now, we turn to the case when $\xi\colon \bZ \to \bC$ is a stationary sequence of random variables
(formally, we need only the restriction of $\xi$ on $\bZ_+$, but due to stationarity, this restriction
determines a unique extension of $\xi$ onto $\bZ$). As usual,
stationarity means that, for every positive integer $k$, every choice of integers $n_1$, ..., $n_k$, and every integer $m$,
the $k$-tuples of random variables
\[
\bigl\langle \xi (n_1), ...\,, \xi (n_k) \bigr\rangle\,, \quad
\bigl\langle \xi (n_1+m), ...\,, \xi (n_k+m) \bigr\rangle
\]
are equidistributed. In what follows, we deal only with stationary sequences having
a finite second moment. Then the sequence
\[
m \mapsto \bE \bigl\{ \xi (0)\, \overline{\xi (m)} \bigr\}
\]
is positive-definite, and therefore, is the Fourier transform of a non-negative
measure $\rho\in M_+(\bT)$. Here and elsewhere,
$\bT = \{ e^{\im\theta}\colon |\theta|\le\pi \}$ is the unit circle.
We call $\rho$ {\em the spectral measure}
of $\xi$. Then {\em the spectrum} $\sigma (\xi)$ of $\xi$ is the support of
the measure $\rho$.
Note that we do not require that $\bE\xi(0) =0$. The definition of the spectral measure
we use here differs from the one, which is more customary in the theory of stationary
processes~\cite{IL}, by the atom at $\theta=0$ with the mass $|\bE \xi (0)|^2$.

\medskip
We also need
{\em the maximal correlation coefficient} of the sequence
$\xi$
\[
r(m) =r_{\xi}(m) \stackrel{\rm def}=
\sup \Biggl\{ \frac{\bigl| \bE \{ (x-\bE x)\overline{(y-\bE y)} \}\bigr|}{\sqrt{\bE |x-\bE x|^2 \cdot \bE |y-\bE y|^2}}\colon
x\in L^2_{(-\infty, 0]}, y\in L^2_{[m, +\infty)} \Biggr\}
\]
where $L^2_{(-\infty, 0]}$ is the space of the random variables measurable with respect to the $\sigma$-algebra
generated by the set $\bigl\{ \xi (n)\colon -\infty<n\le 0 \bigr\}$ with finite second moment, and
$L^2_{[m, +\infty)}$ is the space of the random variables measurable with respect to the $\sigma$-algebra
generated by the set $\bigl\{ \xi (n)\colon m \le n < +\infty \big\}$ with finite second moment.

\begin{theorem}\label{thm:stat-corr}
Let $\xi$ be a bounded stationary sequence of random variables, and let the maximal
correlation coefficient of $\xi$ satisfy
\begin{equation}\label{eq:corcoeff}
r(m) = O\bigl( (\log m)^{-\kappa}\bigr)\,, \qquad m\to\infty\,,
\end{equation}
with some $\kappa>1$. Then, almost surely, $\xi$ is an $L$-sequence.
\end{theorem}

\subsection{}
Now, we turn to the Gaussian stationary sequences $\xi$. In this case,  the leading term of the asymptotics of $\log|F_\xi|$ is determined by the support of
the spectral measure $\rho$ of the sequence $\xi$, see Section~\ref{sec10}.

We start with some preliminaries. For any set
$\sigma\subset\bT$, we denote by ${\rm ch}(\sigma)$ the closed convex hull of $\sigma$,
and by
\[
H_{\sigma} (z) \stackrel{\rm def}= \max_{\la\in {\rm ch}(\sigma)} {\rm Re\,} \bigl( z \bar\la \bigr)
= \sup_{\la\in\sigma}{\rm Re\,} \bigl( z \bar\la \bigr)
\]
the ``Minkowski functional'' of ${\rm ch}(\sigma)$.
This function is subharmonic in $\bC$ and homogeneous, that is,
$ H_\sigma ( re^{\im\theta} ) = h_\sigma (\theta) r $,
where $h_\sigma$ is the so called {\em supporting function} of ${\rm ch}(\sigma)$.
The distributional Laplacian of the function $H_\sigma$ is
$ \Delta H_\sigma = \rd r \otimes \rd s_\sigma (\theta) $,
where $\rd s_\sigma (\theta) = (h_\sigma^{''} + h_\sigma)\, \rd\theta$, the second derivative
$h_\sigma^{''}$ is also understood in the sense of distributions.

\medskip\noindent{\bf Definition.}
Let $\sigma\subset\bT$. We say that the sequence $\xi$ is an $L(\sigma)$-sequence,
if
\begin{equation}\label{eq:log-as-general}
\frac{\log |F_\xi (tz)|}{t} \underset{t\to\infty}\longrightarrow H_\sigma (z)\,, \qquad {\rm in\ }L^1_{\rm loc}(\mathbb C)\,.
\end{equation}
\noindent
Obviously, $L$-sequences are a special case of $L(\sigma)$-sequences that  correspond to the case when
the set $\sigma$ is dense in $\bT$.

\medskip
In the language of the entire function theory \cite[Chapters~II and~III]{Levin}
(see also~\cite{Azarin}  for a modern treatment), this definition says that
$F_\xi$ is an entire function of completely regular growth in the Levin--Pfluger sense
with the Phragm\'en--Lindel\"of indicator $h_\sigma$.
Condition~\eqref{eq:log-as-general} yields the angular asymptotics of zeroes of $F_\xi$:
\begin{equation}\label{eq:zeroes-ang-dens}
n_{F_\xi} (r; \theta_1, \theta_2) = \frac{\bigl( s_\sigma (\theta_2) - s_\sigma (\theta_1) + o(1) \bigr)\, r}{2\pi}\,,
\qquad r\to\infty\,,
\end{equation}
where $-\pi\le \theta_1<\theta_2\le\pi$
with at most countable set of exceptional values of $\theta_1$ and $\theta_2$ that correspond
to possible atoms of the measure $s_\sigma$, cf.~\eqref{eq:zeroes-as'}.
It also  yields the Lindel\"of-type symmetry condition, namely, the existence of the limit
\begin{equation}\label{eq:Lindelof}
\lim_{r\to\infty}\, \sum_{|z_n|\le r} \frac1{z_n}\,,
\end{equation}
where the sum is taken over zeroes of $F_\xi$. In the reverse direction, for functions of exponential type, conditions~\eqref{eq:zeroes-ang-dens}
and~\eqref{eq:Lindelof} together yield~\eqref{eq:log-as-general}.

\medskip We say that the stationary sequence $\xi$ is Gaussian, if $\bigl({\rm Re\,}\xi(n), {\rm Im\,}\xi (n) \bigr)$ are
random normal vectors in $\bR^2$ with non-zero covariance matrix (so that this definition includes also real-valued Gaussian stationary sequences).
\begin{theorem}\label{thm:stat-gauss}
Suppose $\xi$ is a Gaussian stationary sequence with the spectrum $\sigma=\sigma (\xi)$.
Then, almost surely, $\xi$ is an $L(\sigma^*)$-sequence, where $\sigma^*$ is the reflection of $\sigma$
in the real axis.
\end{theorem}

Comparing this result with Theorem~\ref{thm:stat-corr}, we note
that if $\xi$ satisfies condition~\eqref{eq:corcoeff}, then the spectral measure $\rho$
has a density $|f|^2$, where $f$ belongs to the Hardy space $H^2(\bT)$. This follows from a
classical result that goes back to Kolmogorov, see \cite[Chapter~XVII, \S~1]{IL}.
Since no function in $H^2(\bT)\setminus\{0\}$ vanishes on an arc, 
we obtain that for every Gaussian stationary sequence $\xi$ satisfying \eqref{eq:corcoeff} we have $\sigma (\xi) = \bT$, and, 
almost surely, $\xi$ is an $L$-sequence.

\subsection{}
Theorems~\ref{thm:stat-corr} and ~\ref{thm:stat-gauss} have a counterpart for uniformly
almost-periodic sequences found by Levin~\cite[Chapter~VI, \S~7]{Levin}, which we will
recall here.

Let $\xi\colon \bZ\to \bC$ be a uniformly almost-periodic sequence, that is, a uniform limit of trigonometric polynomials on $\mathbb Z$.
Then the limit
\[
\widehat\xi(e^{\im\la}) = \lim_{N\to\infty} \frac1{2N+1}\, \sum_{|n|\le N} \xi (n) e^{-\im\la n}
\]
exists for every $e^{\im\la}\in\bT$, and does not vanish for a non-empty at most countable
set of $e^{\im\la}$. This set is called {\em the spectrum} of $\xi$, and the values $\widehat\xi(e^{\im\la})$ are called the Fourier coefficients of 
$\xi$ (compare to \cite[Section VI.5]{Katznelson}).

\begin{theorem}[B. Ya. Levin]\label{thm:Levin}
Suppose $\xi$ is a uniformly almost-periodic sequence with the spectrum $\sigma$. Then
$\xi$ is an $L(\sigma^*)$-sequence,
with $\sigma^*$ being the reflection of $\sigma$ in the real axis.
\end{theorem}

The proof of this theorem given in~\cite{Levin} is based upon deep results on the zero distribution
of entire functions approximated by finite linear combinations of exponents. For the reader's convenience,
we include the proof of this theorem, which is based on the same ideas as Levin's original proof, but can be
read independently of the theory developed in~\cite[Chapter~VI]{Levin}.

\subsection{}
Here, we briefly explain how Theorems~\ref{thm:Q}--\ref{thm:Levin}
are related to a wealth of results, which deal with the analytic continuation of the
Taylor series
\[
f_\xi (s) \stackrel{\rm def}{=}\sum_{n\ge 0} \xi( n) s^n
\]
through the boundary of the disk of convergence. A survey of these results obtained prior to 1955
can be found in~\cite{Bieberbach}. First, observe that the function
\[
w^{-1} f_\xi (w^{-1}) = \sum_{n\ge 0} \frac{\xi (n)}{w^{n+1}}
\]
is the Laplace transform of $F_\xi$. Then, by P\'olya's theorem (see~\cite[Theorem~33, Chapter~I]{Levin}
or~\cite[Theorem~1.1.5]{Bieberbach}), the upper
limit
\begin{equation}
H^{F_\xi}(z) = \limsup_{t\to\infty} \frac{\log |F_\xi (tz)|}t
\label{dop}
\end{equation}
is the Minkowski functional of the closed convex hull of the set of singularities of
the function $w^{-1} f_\xi (w^{-1})$, reflected in the real axis. Hence, the results
about analytic continuation of $f_\xi$ provide information about {\em the upper limit} in \eqref{dop} 
but not about {\em the existence of the limit} in~\eqref{eq:log-as-general}.

For instance, the property that the unit circle is a natural boundary for the Taylor series
$f_\xi$ is {\em equivalent} to the property that the upper limit $H^F(z) \equiv 1$, but it
cannot guarantee that $\xi$ is an $L$-sequence.

\subsection{}
Here, we mention two curious results, which follow from Lemma~\ref{lemma8.2}
and which might be of an independent interest.
 
\subsubsection{}
The first result sheds some light on the nature of very strong cancellations in Taylor series:
{\em Suppose $\xi\colon\bZ\to\bC$ is a stationary sequence with the spectral measure $\rho$. Then,
almost surely,}
\[
\limsup_{r\to\infty} \frac{\log|F_\xi (re^{{\rm i}\theta})|}r \le \max_{t\in\operatorname{spt}(\rho)}\,
\cos (\theta+t)\,.
\]
In particular, $ F_\xi (re^{{\rm i}\theta}) $, {\em almost surely}, exponentially decays on some angle $A=\{z:\theta_1<\arg(z)<\theta_2\}$
provided that the origin does not belong to the convex hull ${\rm ch}(\sigma)$ of the support of $\rho$. Just choose $A\subset {\rm ch}(\sigma)^O$. 
Here $C^O$ is the polar cone of a plane set $C$, $C^O=\{z\in\mathbb C:{\rm Re\,} z\bar w\le 0,\, w\in C\}$.

Note that this result is helpful when there are no special restrictions on the support of the spectral measure.

\subsubsection{}\label{subsubsect:Szego}
The second result says that in some situations such restrictions do exist.
We say that a set $A\subset\mathbb C$ is uniformly discrete if
$\inf\bigl\{ |z-w|\colon z, w\in A, z\ne w \bigr\}\!>\!0$.

\begin{theorem}\label{thm:6}
Suppose that $\xi\colon\bZ \to \bZ$ is a stationary integer-valued sequence.
Let $\rho$ be the spectral measure of $\xi$. Then either $\operatorname{spt}(\rho)=\bT$, or
the sequence $\xi$ is periodic and $\operatorname{spt}(\rho)\subset \{w:w^N=1\}$ for some $N\ge 1$.

\smallskip The same conclusion holds if $\xi\colon\bZ \to A$ is 
an ergodic stationary sequence and the set $A$ is uniformly discrete.
\end{theorem}

\section{Subharmonic preliminaries}

\subsection{}
In this section we will systematically use the following facts on the local convergence of subharmonic functions, see, for instance, 
\cite[Theorem~4.1.9]{Hor} and \cite[Theorem~3.2.12]{HorC}. 

\begin{theorem1}\label{thm:A} Let $(v_j)_j$ be a sequence of subharmonic functions on the plane having a uniform upper bound on any compact set. 
Then 

{\rm (a)} if $(v_j)$ does not converge to $-\infty$ uniformly on every compact set, then there is a subsequence $(v_{j_k})$ converging 
in $L^1_{\rm loc}(\mathbb C)$;

{\rm (b)} if $v_j\not\equiv-\infty$ for every $j$, and $v_j\to U$ in $L^1_{\rm loc}(\mathbb C)$, then 
$U$ is equal almost everywhere to a subharmonic function;

{\rm (c)} if $v$ is a subharmonic function and $v_j\to v$ in $L^1_{\rm loc}(\mathbb C)$, then 
$$
\text{\rm (i)} \qquad \limsup_{j\to\infty}v_j(z)\le v(z),\qquad z\in\mathbb C, 
$$
\qquad\quad with the two sides equal and finite almost everywhere, and 
$$
\hskip -1.65 cm\text{\rm (ii)} \qquad \limsup_{j\to\infty}\,\sup_K v_j\le \sup_K v
$$
\qquad\quad for every compact set $K$ on the plane.
\end{theorem1}

Now, we recall several basic facts from Azarin's theory of limit sets
of subharmonic functions~\cite{Azarin}.
In what follows, we deal only with entire functions $F$ of exponential type.
That is, $ |F(z)| \le A e^{\tau |z|}$, $z\in\bC$.
Consider a family of subharmonic functions
\[
u_t(z) = \frac1{t}\, \log |F(tz)|\,, \qquad t\ge 1\,.
\]
By Proposition~A, this family is pre-compact in $L^1_{\rm loc}(\mathbb C)$.
For every $L^1_{\rm loc}(\mathbb C)$-limit $U$ of subharmonic functions $u_{t_k}$ there is a unique subharmonic function $u$ 
such that $U=u$ almost everywhere. We remark that $u$ might not be the pointwise limit of the 
$u_{t_k}$ (for instance, this limit might fail to be upper semi-continuous). 
 Now, each sequence $t_j\to\infty$ has a subsequence $t_{j_k}$ such that $u_{t_{j_k}}$ converges
in $L^1_{\rm loc}(\mathbb C)$ to a subharmonic function $v$. By $\mathcal L (F)$ we denote the set of all limiting subharmonic functions
$v$. The set $\mathcal L (F)$ is called {\em the limit set} of $\log |F|$. This set is invariant with respect to
the multiplicative action of $\bR_+$, that is, if $v\in \mathcal L(F)$, then for each $t>0$,
\begin{equation}\label{eq:3.1}
{\rm the\ function}\quad  v_t(z) = t^{-1} v(tz) \quad  {\rm also\ belongs\ to\ } \mathcal L(F)\,.
\end{equation}
Since $F$ is an entire function of exponential type, every function $v\in \mathcal L(F)$ satisfies
\[
v(z) \le \tau\,|z|\,, \qquad z\in\bC\,.
\]
%and $v(0)=0$.

\medskip {\em The homogeneous indicator} $H^F$ of $F$ is the upper envelope of functions in $\mathcal L(F)$:
\[
H^F(z) \stackrel{\rm def}= \sup_{v\in \mathcal L(F)}\, v(z)\,, \qquad z\in\bC\,.
\]
Then $H^F(re^{\im\theta})= h^F(\theta)\,r$, where
\[
h^F(\theta) = \sup_{v\in \mathcal L(F)}\, v(e^{\im\theta})\,, \qquad -\pi \le\theta\le \pi
\]
is {\em the Phragm\'en--Lindel\"of indicator} of $F$. An equivalent (more traditional) definition of $h^F$ is
\[
h^F(\theta) = \limsup_{r\to\infty}\, \frac{\log |F(re^{\im\theta})|}{r}\,.
\]

In particular, this definition gives that $h^F$ is continuous, see \cite[Chapter I, Sections 15,16]{Levin}. To verify the equivalence, we need to check that 
for every $\theta$, 
$$
A\stackrel{\textrm{def}}{=}\sup_{v\in \mathcal L (F)}v(e^{\im\theta})=\limsup_{r\to\infty}u_r(e^{\im\theta})\stackrel{\textrm{def}}{=}B.
$$
First, we can choose a subharmonic function $v$ and a sequence $r_k\to\infty$ such that 
\begin{gather*}
\lim_{k\to\infty}u_{r_k}(e^{\im\theta})=B,\\
u_{r_k}\to v \text{\ \ in\ \ }L^1_{\rm loc}(\mathbb C).
\end{gather*}
By Proposition~A,
$$
\lim_{k\to\infty}u_{r_k}(e^{\im\theta})\le v(e^{\im\theta})\le A,
$$
and we conclude that $A\ge B$.
In the opposite direction, if $B<A$, then again by Proposition~A there exist a subharmonic function $v$, a sequence $r_k\to\infty$, and 
a neighbourhood $U$ of $e^{\im\theta}$ such that 
$$
\sup_{z\in U}\limsup_{k\to\infty}u_{r_k}(z)<v(e^{\im\theta})
$$
and for almost all $z$ in $U$,
$$
\limsup_{k\to\infty}u_{r_k}(z)=v(z).
$$
This contradicts to the subharmonicity of $v$.

The homogeneous indicator $H^F$ is the Minkowski functional of a convex compact set called {\em the indicator
diagram} $I^F$ of $F$, see \cite[Chapter I, Section 19]{Levin}.

\medskip The ray $\bigl\{ \arg (z) = \theta \bigr\}$ is called {\em a  ray of completely regular growth}
of the function $F$ if the set $\mathcal L(F)$ restricted on that ray is singleton. Then
\begin{equation}\label{eq:c.r.g.-ray}
v(re^{\im\theta}) = H^F (re^{\im\theta}) = h^F(\theta)\, r\,, \qquad v\in \mathcal L(F)\,.
\end{equation}
By continuity of the Phragm\'en--Lindel\"of indicator, {\em the set of the rays of completely regular growth is
closed}. Clearly, the function $F$ has completely regular growth in $\bC$ if it has a completely regular growth
on every ray. Hence, it suffices to verify condition~\eqref{eq:c.r.g.-ray} on a dense set of rays.

\subsection{}

\noindent{\bf Definition.}
We say that a sequence $R_j\uparrow \infty$ is {\em thick} if
$\displaystyle
\lim_{j\to\infty}
\frac{R_{j+1}}{R_j} = 1$.

\begin{lemma}\label{lemma-subh}
Let $F$ be an entire function of exponential type. Let $h^F(\theta)\le\kappa$ for some
$\theta\in [-\pi, \pi]$. Suppose that there exist a thick sequence $R_j\uparrow\infty$ and a sequence
$\theta_j\to\theta$ such that
\begin{equation}\label{eq:3.3}
\liminf_{j\to\infty}\, \frac1{R_j}\, \log|F(R_je^{\im\theta_j})| \ge \kappa\,.
\end{equation}
Then $h^F(\theta)=\kappa$ and $F$ has completely regular growth on the ray
$\{\arg(z)=\theta\}$.
\end{lemma}

\begin{proof}
Suppose that there exists a function $v\in\mathcal L(F)$ such that
$v(e^{\im\theta})<\kappa$. Since $v$ is subharmonic, in particular upper semi-continuous, in a small compact neighbourhood $U$ of $e^{\im\theta}$ 
we have $\sup_{U}v<\kappa$. Next, 
\[
\frac1{t_k}\, \log|F(t_kz)| \to v(z)\, \qquad {\rm in\ } L^1_{\rm loc}(\mathbb C)
\]
for some sequence $t_k\to\infty$.
By %the upper semi-continuity of subharmonic functions, both point-wise and
%with respect to the $\mathcal D'$-convergence (see, for instance, 
Proposition~A 
we obtain 
\[
\limsup_{k\to\infty}\sup_{z\in U}\, \frac1{t_k}\, \log |F(t_k z)| < \kappa\,,
\]
and hence
\[
\limsup_{k\to\infty}\, \frac1{t_k}\, \log |F(t_k z_k)|% \le v(e^{\im\theta}) 
< \kappa\,,
\]
provided that  $ z_k\to e^{\im\theta}$.

Now, we choose $j_k$ so that $R_{j_k} \le t_k < R_{j_k+1}$, and put
$\tau_k = t_k^{-1} R_{j_k}$ and $ z_k = \tau_k e^{\im\theta_k}$.
Then $\tau_k\to 1$ (this is the place where we use thickness of the sequence
$R_j$), and therefore, $z_k\to e^{\im\theta}$. Thus,
\[
\limsup_{k\to\infty} \frac1{R_{j_k}}\, \log |F(R_{j_k}e^{\im\theta_k})| =
\limsup_{k\to\infty} \frac1{\tau_k t_k}\, \log |F(t_k z_k )| < \kappa\,,
\]
arriving at a contradiction.
\end{proof}

\subsection{}
The following lemma is a variation on the theme of the
maximum principle.
It will be needed for the proof of Theorem~\ref{thm:Levin}.

\begin{lemma}\label{lemma2}
Let $F$ be an entire function of exponential type, and let
$\sigma\subset\bT$. Suppose that

\smallskip\par\noindent{\rm (i)} $h^F \le h_\sigma $ everywhere on $[-\pi, \pi]$;

\smallskip\par\noindent{\rm (ii)}  $h^F = h_\sigma=1 $ everywhere on $\sigma$;

\smallskip\par\noindent{\rm (iii)} $F$ has completely regular growth on the set of rays
$\{z\colon \arg (z)\in\sigma \}$.

\smallskip\par\noindent
Then $h^F = h_\sigma$ everywhere, and $F$ has completely regular growth in $\bC$,
\end{lemma}

\begin{proof}
If $\sigma$ is dense on $\bT$, then the statement is obvious. So we will concentrate
on the case when $\sigma$ is not dense in $\bT$.

Let $I^F$ be the indicator diagram of $F$. By condition (i), $I^F\subseteq {\rm ch}(\sigma)$.
By the definition of the convex hull, ${\rm ch}(\sigma)$ is the smallest convex compact that
contains the set $\sigma$. By condition (ii), $\sigma\subseteq I_F$. Hence, $I^F={\rm ch}(\sigma)$,
that is, $h^F=h_\sigma$ everywhere.

Let $S=\{\theta\colon h^F(\theta)<1 \}$. The set $S$ is a union of disjoint open intervals,
let $J = (\alpha, \beta)$ be one of them. That is,
$h_\sigma (\alpha) = h_\sigma (\beta)=1$, while $h_\sigma <1$ everywhere on $(\alpha, \beta)$.
For $\theta\in\bar J$, we have
\begin{align*}
h^F(\theta) &= \max\bigl( \cos(\theta-\alpha), \cos(\theta-\beta) \bigr) \\
&=
\begin{cases}
\cos(\theta-\alpha), & \alpha\le\theta\le \frac12 (\alpha+\beta), \\
\cos(\theta-\beta),  & \frac12 (\alpha+\beta) \le\theta\le \beta\,.
\end{cases}
\end{align*}

Consider the angle $\alpha\le\arg(z)\le \frac12(\alpha+\beta)$.
In this angle the indicator $h^F$ is trigonometric, and $F$ has a completely
regular growth on the boundary ray $\arg(z)=\alpha$. Moreover,
$(h^F)'(\alpha+0) = 0$ and $(h^F)'(\alpha-0) = 0$. The first relation is
obvious. To see that the second relation holds, we consider two cases:
(i) $\alpha$ is not an isolated point of $[-\pi, \pi]\setminus S$, and
(ii) $\alpha$ is an isolated point of $[-\pi, \pi]\setminus S$.
In the first case, there is a sequence $\theta_\ell \uparrow\alpha$ such that
$h^F(\theta_\ell) = h^F(\alpha)=1$. On each interval $(\theta_\ell, \theta_{\ell+1})$,
we have
\[
0 \le 1 - h^F (\theta) \le O\bigl( (\theta_{\ell+1}-\theta)^2 \bigr)
\le O\bigl( (\alpha-\theta)^2 \bigr).
\]
Hence, $(h^F)'(\alpha-0)=0$. In the second case, this relation is obvious, since
$\alpha$ is a maximum point of a trigonometric function.
Thus, the indicator $h^F$ is $C^1$-smooth at $\theta=\alpha$, and we are  in
the assumptions of Levin's theorem on entire functions with
Phragm\'en--Lindel\"of indicator~\cite[Theorem~7, Chapter~III]{Levin}.
By this theorem,
$F$ has completely regular growth in the angle
$\bigl\{ \alpha \le\theta \le \tfrac12 (\alpha+\beta)\bigr\}$.
Similarly, $F$ has completely regular growth in the angle
$\bigl\{ \tfrac12 (\alpha+\beta) \le \theta \le \beta \bigr\}$.
This proves Lemma~\ref{lemma2}.
\end{proof}

\medskip It is worth mentioning that Levin's theorem used in the proof of Lem\-ma~\ref{lemma2}
can be deduced from Hopf's boundary maximum principle for non-positive subharmonic functions
vanishing on a part of the boundary.

\section{Exponential sums}

\subsection{}
For a bounded sequence $\xi\colon \bZ_+\to \bC$, introduce the exponential sum
\[
W_R(\theta) = \sum_{|n|\le N} \xi (n+R) e(n\theta) e^{-\frac{n^2}{2R}}\,,
\]
where $R$ and $N$ are {\em large integer parameters} such that $N=R^{1/2}\log R + O(1)$. 
In principle, any choice of $N$ in the range
$R^{\frac12}\sqrt{\log R} \ll N \ll_\epsilon R^{\frac12 + \epsilon}$ would suffice for our purposes. Here $x(R)\ll_\epsilon y(R,\epsilon)$ means that for every 
$\epsilon>0$, $x(R)=o(y(R,\epsilon))$ as $R\to\infty$.

\begin{lemma}\label{lemma3}
Let
\[
F_\xi (z) = \sum_{n\ge 0} \xi (n) \frac{z^n}{n!}
\]
with a bounded sequence $\xi\colon \bZ_+\to\bC $. Then,
for each $\epsilon>0$,
\[
|F_\xi (Re(\theta))| \ge \mu (R) \bigl[ |W_R(\theta)| - C_\epsilon R^\epsilon\bigr]\,,
\quad {\rm where\ }
\mu (R) = \frac{e^R}{\sqrt{2\pi R}}\,.
\]
\end{lemma}

\begin{proof}
First, we estimate the tails
\[
\Bigl( \sum_{0\le n < R-N} + \sum_{n> R+N} \Bigr) |\xi (n)| \frac{R^n}{n!}\,.
\]
Put $N_1=R-N$, $N_2=R+N$. These sums are bounded by
\[
O(1)\, \sum_{0\le n\le N_1-1} \frac{R^n}{n!}
\quad {\rm and} \quad
O(1)\, \sum_{n\ge N_2+1} \frac{R^n}{n!}
\]
correspondingly. Note that the sequence $n\mapsto \frac{R^n}{n!}$ increases for
$0\le n \le N_1-1$ and decreases for $n\ge N_2+1$. For $0\le n \le N_1-1$, we have
\[
\frac{R^n}{n!}\, :\, \frac{R^{n+1}}{(n+1)!} = \frac{n+1}R \le 1 - \frac{N}{R}\,,
\]
while, for $n \ge N_2+1$,
\[
\frac{R^{n+1}}{(n+1)!}\, :\, \frac{R^n}{n!} = \frac{R}{n+1} < \frac{R}{N_2} = \frac1{1+\frac{N}{R}}\,.
\]
Whence,
\[
\sum_{0\le n\le N_1-1} \frac{R^n}{n!} < \frac{R^{N_1}}{N_1!}\, \frac1{1-\left( 1 - \frac{N}R \right)}
= \frac{R}N \cdot \frac{R^{N_1}}{N_1!}
\]
and
\[
\sum_{n\ge N_2+1} \frac{R^n}{n!} < \frac{R^{N_2}}{N_2!} \, \frac1{1-\frac1{1+\frac{N}R}}
< 2\, \frac{R}N \cdot \frac{R^{N_2}}{N_2!}\,.
\]
It remains to observe that each of the quantities $\frac{R^{N_1}}{N_1!}$ and $\frac{R^{N_2}}{N_2!}$ does not exceed
$Ce^{-c R^{-1}N^2}\, \mu (R)$, provided that $\sqrt{R} \ll N \ll R$.
Therefore,
\[
F_\xi (Re(\theta)) = \sum_{|n-R|\le N}\, \xi (n) e(n\theta)\, \frac{R^n}{n!} + O(1) \mu (R)\,,
\]
provided that $\sqrt{R \log R} \ll N \ll R$.

Now, we turn to the central group of terms of the series.
By Stirling's formula, we have
\begin{multline}\label{eq:4.2}
\sum_{|n-R|\le N}\, \xi (n) e(n\theta)\, \frac{R^n}{n!}
= \mu(R)\,  \sum_{|n-R|\le N}\, \xi (n) e(n\theta)\, \frac{R^n}{n!} \cdot \frac{\sqrt{2\pi R}}{e^R}
\\
= \mu (R)\,  \sum_{|n-R|\le N}\, \xi (n) e(n\theta)\, \bigl( 1 + O(R^{-1}) \bigr)\, \Bigl( \frac{R}n\Bigr)^{n+\frac12}\, e^{n-R}\,.
\end{multline}
Put $t=n-R$. Then $|t|\le N$, and
\begin{align*}
\Bigl( \frac{R}n\Bigr)^{n+\frac12} e^{n-R}
&= \exp\Bigl( \bigl( R+t+\tfrac12 \bigr) \log\bigl( 1-\frac{t}{R+t}\bigr) +t \Bigr) \\
&=\exp \Bigl( -\frac{t^2}{2(R+t)} - \frac{t}{2(R+t)} + O\Bigl( \frac{|t|^3}{R^2} \Bigr) \Bigr) \\
&=\exp  \Bigl( -\frac{t^2}{2R} + O\Bigl( \frac{|t|}{R}\Bigr) + O\Bigl( \frac{|t|^3}{R^2} \Bigr) \Bigr) \\
&= \exp\Bigl( -\frac{t^2}{2R} + O\Bigl( \frac{N^3}{R^2} \Bigr) \Bigr) \\
&= \Bigl( 1 + O\bigl( R^{-\frac12 + 3\epsilon}\bigr) \Bigr) e^{-\frac{t^2}{2R}}\,.
\end{align*}
Hence, the sum on the RHS of~\eqref{eq:4.2} equals
\[
\mu (R)\, \sum_{|n-R|\le N}\, \xi (n) e(n\theta)\, e^{-\frac1{2R} (n-R)^2} + \Omega\, \mu (R)
\]
with
\[
|\Omega| \le O(1)\, N \cdot R^{-\frac12+3\epsilon} = O\bigl( R^{4\epsilon} \bigr)\,.
\]
This completes the proof of Lemma~\ref{lemma3}.
\end{proof}

\subsection{}
Combining Lemmas~\ref{lemma-subh} and~\ref{lemma3}  
we arrive at
\begin{lemma}\label{lemma4}
Let
\[
F_\xi (z) = \sum_{n\ge 0} \xi (n) \frac{z^n}{n!}\,,
\]
where $\xi\colon \bZ_+\to\bC $ is a bounded sequence.
Suppose that for every $a\in [0, 1]$ there exist
a thick sequence $R_j\uparrow\infty$, a sequence $\theta_j\to a$, and
$\delta>0$ so that
\begin{equation}\label{eq:4.3}
\bigl| W_{R_j} (\theta_j) \bigr| \ge R_j^\delta\,.
\end{equation}
Then $\xi$ is an $L$-sequence.
\end{lemma}

\subsection{}
In many instances it is easier to produce a lower bound for an average of
$|W_R|^2$ over short intervals of $\theta$. The following lemma is a
straightforward corollary to the previous one.

From now on, we fix a non-negative even function $g\in C^2_0[-\tfrac12, \tfrac12]$ with
$\int g(\theta)\,\rd\theta =1$.

\begin{lemma}\label{lemma4.5}
Let
\[
F_\xi (z) = \sum_{n\ge 0} \xi (n) \frac{z^n}{n!}\,,
\]
where $\xi\colon \bZ_+\to\bC $ is a bounded sequence. Suppose that,
for every $a\in [0, 1]$ and for every $m\in \bN$, there
exist a thick sequence $R_j\uparrow\infty$ and $\delta>0$ so that
\begin{equation}\label{eq4.4}
\int_{a-\frac1{2m}}^{a+\frac1{2m}} \bigl| W_{R_j}(\theta) \bigr|^2
g(m(\theta-a))\, \rd\theta \ge R_j^\delta\,, \qquad j\ge j_0(a, m)\,.
\end{equation}
Then $\xi$ is an $L$-sequence.
\end{lemma}

Curiously enough, assumptions of Lemmas~\ref{lemma4} and~\ref{lemma4.5} impose restrictions
only on relatively short blocks
$ \bigcup_j \bigl[ R_j - R_j^{\frac12+\epsilon}, R_j + R_j^{\frac12+\epsilon}\bigr] $
of elements of the sequence $\xi$. The values attained by $\xi$ off these blocks do not matter.

\section{Proof of Theorems~1 and~2 ($\beta>\frac32$)}

In this part, we put $\xi (n) = e(f(n))$ for some real-valued $f$. Then
\[
W_R(\theta) = \sum_{|n|\le N} e( f(n+R)+n\theta)\, e^{-\frac{n^2}{2R}}\,,
\qquad \sqrt{R\log R} \ll  N \ll R^{\frac12 +\epsilon}\,,
\]
and we are looking for a lower bound for
\[
X_R = \int_{a-\frac1{2m}}^{a+\frac1{2m}} \bigl| W_R(\theta) \bigr|^2
g(m(\theta-a))\, \rd\theta\,, \qquad
a\in [0, 1], \ m\in\bN\,.
\]
The upper bound $X_R \le C\sqrt{R}$ as well as the
matching lower bound in the case when $m=1$ follow from Parseval's theorem.
There are some reasons to expect that if there are no unreasonable
cancellations, then a similar lower bound holds in all scales, that is,
$X_R \ge c(a, m)\sqrt{R}$ for every $m\in\mathbb N$ and every
$a\in [0, 1]$. In the next sections, we justify these expectations.

\subsection{}
The following lemma reduces the lower bound for $X_R$ to upper bounds
for certain Weyl sums. Put
\[
S_T(M_1, M_2) = \sum_{M_1 \le n < M_2} e\bigl( f(n+R)-f(n+R-T) \bigr)\,.
\]
\begin{lemma}\label{lemma5}
There exist positive numerical constants $c$ and $C$ so
that
\[
X_R \ge \frac{c\sqrt{R}}{m} -  Cm\, \sum_{T=1}^{2N} \frac1{T^2}\
\max_{\substack{0< M_2-M_1 \le \sqrt{R},\\|M_1|, |M_2|\le N}}\ \bigl| S_T(M_1, M_2) \bigr|\,.
\]
\end{lemma}

\begin{proof}
We have
\begin{multline*}
X_R =
\frac1{m}\, \int_{-1/2}^{1/2} \Bigl| \sum_{|n|\le N}\, e \Bigl( f(n+R) + na + \frac{n\theta}{m} \Bigr) e^{-n^2/(2R)}\Bigr|^2\,
g(\theta)\, \rd\theta \\
= \frac1{m} \, \sum_{|n|, |n'| \le N}\,
e\bigl( f(n+R)-f(n'+R) + (n-n')a\bigr) e^{-(n^2+ n'^2)/(2R)}\, \widehat{g}\bigl( \frac{n'-n}m \bigr)\,,
\end{multline*}
where $\widehat{g}$ denotes the Fourier transform of $g$ extended by $0$ to $\mathbb R\setminus [-\tfrac12,\tfrac12]$. The diagonal sum ($n=n'$) contributes
\[
\frac1{m}\, \sum_{|n|\le N} e^{-n^2/R}\, \stackrel{N\ge\sqrt{R}}\ge\, \frac{c\sqrt{R}}{m}\,.
\]
We need to estimate from above the contribution of non-diagonal terms
\[
\frac2m \Bigl| \sum_{\substack{|n|, |n'|\le N \\ n'<n}}\, e\bigl( f(n+R)-f(n'+R) + (n-n')a \bigr)\, e^{-(n^2+n'^2)/(2R)}\,
\widehat{g}\bigl( \frac{n'-n}{m} \bigr) \Bigr|\,.
\]
Letting $T=n-n'$ and using that
\[
\widehat{g}\bigl( \frac{n'-n}{m} \bigr) = O\Bigl(\frac{m^2}{(n-n')^2}\Bigr)\,,
\]
we see that the contribution of non-diagonal terms is
\begin{equation}\label{eq5.a}
\le Cm\, \sum_{T=1}^{2N} \frac1{T^2}\,
\Bigl| \sum_{-N+T\le n \le N}\, e\bigl(f(n+R)-f(n+R-T)\bigr)\, e^{-(n^2+(n-T)^2)/(2R)} \Bigr|\,.
\end{equation}

The function $n\mapsto e^{-(n^2+(n-T)^2)/(2R)}$ increases for $-\infty<n\le \frac12 T$
and decreases for $\frac12 T \le n < +\infty$. We consider these two ranges separately.
Then the expression in~\eqref{eq5.a} is
\[
\le Cm\, \sum_{T=1}^{2N} \frac1{T^2}
\Bigl[ \Bigl| \sum_{-N+T\le n < \frac12 T}\, ... \, \Bigr| +
\Bigl| \sum_{\frac12 T \le  n \le N}\, ... \, \Bigr|
\Bigr]\,.
\]
Next, we split the sums in $n$ into blocks of length $\sqrt{R}$
(and several blocks of smaller length that are treated similarly). We set $J_k=[\frac{T}2+(k-1)\sqrt{R},\frac{T}2+k\sqrt{R})$ and
put
\[
Y_{k, T} = \ \sum_{n\in J_k\cap [-N+T,N]
} \
e(f(n+R)-f(n+R-T))   e^{-(n^2+(n-T)^2)/(2R)}\,,
\]
with $|k|\le \frac1{\sqrt R}\, \bigl( N-\frac{T}2 \bigr)+1$. Note that for
$n=\frac{T}2+\la\sqrt{R}$ with $k-1\le\la<k$, we have
\begin{align*}
-\frac1{2R} \bigl( n^2 + (n-T)^2 \bigr) &=
-\frac1{2R} \Bigl( \bigl( \la\sqrt{R} + \frac{T}2 \bigr)^2
+ \bigl( \la\sqrt{R} - \frac{T}2 \bigr)^2 \Bigr) \\
&=-\frac1{2R} \Bigl( 2\la^2 R + \frac12 T^2 \Bigr) < - \la^2 \le -ck^2 +c_1\,.
\end{align*}
Then applying the Abel summation formula to the sum $Y_{k, T}$, we see that
\[
\bigl| Y_{k, T} \bigr| \le C e^{-ck^2}\ \max_{\substack{0<M_2-M_1\le\sqrt R\\ |M_1|, |M_2|\le N} } \
\bigl| S_T (M_1, M_2) \bigr|\,,
\]
and the sum of non-diagonal terms we are estimating is
\begin{multline*}
\le Cm\, \sum_{T=1}^{2N} \frac1{T^2}\
\sum_{|k|\le \frac1{\sqrt R} (N-\frac12 T)+1}\ \bigl| Y_{k, T} \bigr| \\
\le Cm\, \sum_{T=1}^{2N} \frac1{T^2}\ \max_{\substack{0<M_2-M_1\le\sqrt R\\ |M_1|, |M_2|\le N} } \
\bigl| S_T (M_1, M_2) \bigr|\,.
\end{multline*}
This completes the proof of Lemma~\ref{lemma5}.
\end{proof}

Now, Theorems~1 and~2 (for $\beta>\frac32$) will readily follow from the classical
Weyl and van der Corput estimates of exponential sums.

\subsection{Proof of Theorem~\ref{thm:Q}}

First, we fix $T_0=T_0(m)$ so large that
\[
Cm\, \sum_{T>T_0} \frac1{T^2} < \frac12 \frac{c}{m}\,,
\]
where the positive numerical constants $C$ and $c$ are the same
as in the assertion of Lemma~\ref{lemma5}. Then, using the trivial bound
$\bigl| S_T (M_1, M_2) \bigr| \le \sqrt{R}$, we get
\[
X_R > \frac12 \frac{c}{m}\, \sqrt{R} -
Cm T_0 \, \max_{1\le T \le T_0}\, \max_{0<M_2-M_1 \le \sqrt{R}}\,
\bigl| S_T (M_1, M_2) \bigr|\,.
\]

Define
$$
P_T(x) = Q(x)-Q(x-T),
$$
set
$$
\sum_{k=1}^{d-1} p_k x^k=P_T(x),
$$
and observe that at least one of the coefficients $p_k$ is irrational
(if $\ell$ is the maximal index such that the coefficient $q_\ell$ of $Q$ is irrational,
then $p_{\ell-1}$ must be irrational too). Then, by Weyl's theorem~\cite[Section~3]{Weyl} (see also the argument in \cite[pp. 17--18]{Mon}) we have 
\[
\max_{0<M_2-M_1\le M} \, \Bigl| \, \sum_{M_1\le n < M_2}\,
e(P_T(n)) \Bigr| = o(M)\,, \qquad {\rm as}\quad M\to\infty\,.
\]
Hence, for each $T\in \{ 1, ...\, T_0\}$,
\[
\max_{0\le M_2-M_1\le \sqrt{R}}\, \bigl| S_T(M_1, M_2) \bigr| = o(\sqrt{R})\,,
\qquad {\rm as}\quad R\to\infty\,,
\]
and, for $R>R_0(m)$, we have $X_R > c(m) \sqrt{R}$ with
$c(m)>0$. An application of Lemma~\ref{lemma4.5} completes the proof of Theorem~\ref{thm:Q}.
\hfill $\Box$

\subsection{Proof of Theorem~\ref{thm:beta}}
Here, we prove Theorem~\ref{thm:beta} for $\beta>\frac32$. The case
$\beta=\frac32$ will be treated in Section~\ref{sec6}.
Put $ f_{T, R} (x) = (R+x)^\beta - (R+x-T)^\beta $.

\subsubsection{$\frac32 < \beta < 2$}
In this case, we apply the classical summation formula
(see, for instance,~\cite[(2.1.2)]{Titchmarsh})
\begin{multline*}
\sum_{M_1\le n < M_2}\, \phi (n) = \int_{M_1}^{M_2} \phi (x)\, \rd x
+ \int_{M_1}^{M_2} \bigl( x-[x]-\frac12 \bigr) \phi'(x)\, \rd x \\
+ \frac12\, \phi (M_1)
- \frac12\, \phi (M_2)
\end{multline*}
with $\phi (x) = e(f_{T, R}(x))$ and with integer $M_1$ and $M_2$, $|M_1|, |M_2|\le N$.
We get
\begin{multline*}
S_T(M_1, M_2) = \int_{M_1}^{M_2} e(f_{T, R}(x))\, \rd x \\
+ 2\pi\im\, \int_{M_1}^{M_2} \bigl( x-[x]-\frac12 \bigr) f_{T, R}'(x)\, e(f_{T, R}(x)) \rd x
+ O(1)\,.
\end{multline*}
Since the function $f_{T, R}'$ is monotonically decreasing, applying
a classical estimate on integrals of oscillating functions
(see, e.g., \cite[Lemma~4.2]{Titchmarsh} ), and recalling that $M_2\le N$,
we get
\begin{equation}\label{eq5.b}
\Bigl| \int_{M_1}^{M_2} e(f_{T, R}(x))\, \rd x \Bigr| \le \frac4{f_{T, R}'(N)}\,.
\end{equation}
For $R\ge R_0(\beta)$, $|x|\le N$, and $1\le T \le 2N$, we have
\[
f_{T, R}'(x) = \beta \bigl( (R+x)^{\beta-1} - (R+x-T)^{\beta-1} \bigr)
= \bigl( \beta (\beta-1) + o(1) \bigr)\, \frac{T}{R^{2-\beta}}\,,
\]
uniformly in $x\in [-N, N]$. Therefore, the LHS of~\eqref{eq5.b} is $ \le c(\beta) T^{-1}R^{2-\beta}$.

Next,
\begin{multline*}
\Bigl|
\int_{M_1}^{M_2} \bigl( x-[x]-\frac12 \bigr) f_{T, R}'(x)\, e(f_{T, R}(x) \rd x \Bigr| \\
\le (M_2-M_1)\, \max_{|x|\le N} \bigl| f_{T, R}'(x) \bigr| \le c(\beta) N\, \frac{T}{R^{2-\beta}}\,,
\end{multline*}
whence, by Lemma~\ref{lemma5},
\[
X_R \ge \frac{c\sqrt{R}}{m} - C(\beta)m \bigl( R^{2-\beta} + N R^{\beta-2} \log N \bigr) \ge
c(m, \beta) \sqrt{R}\,,
\]
provided that $R\ge R_0(m, \beta)$. In view of Lemma~\ref{lemma4.5}, this proves Theorem~\ref{thm:beta}
in the case $\frac32 < \beta <2$. \hfill $\Box$

\subsubsection{$\beta>2$}
Suppose that
$k<\beta<k+1$ with an integer $k\ge 2$. To estimate
\[
\max_{\substack{M_2-M_1\le N \\
|M_1|, |M_2|\le N}} \bigl| S_T (M_1, M_2) \bigr|\,,
\]
we apply a van der Corput bound~\cite[Theorem~5.13]{Titchmarsh}. Using that
\[
f_{T, R}^{(k)}(x) \simeq_{\beta, k}\, \frac{T}{R^{k+1-\beta}}
\qquad {\rm uniformly\ in} \quad |x|\le N, \quad 1 \le T \le 2N\,,
\]
we get
\[
\bigl| S_T (M_1, M_2) \bigr| \lesssim_{\beta, k}\,
(M_2-M_1) \Bigl( \frac{T}{R^{k+1-\beta}} \Bigr)^{\frac1{2K-2}}
+ (M_2-M_1)^{1-\frac2{K}} \Bigl( \frac{R^{k+1-\beta}}T \Bigr)^{\frac1{2K-2}}
\]
with $K=2^{k-1}$ ($x\simeq_{a}y$ and $x\lesssim_{a}y$ mean, correspondingly, $c_1(a)y\le x\le c_2(a)y$ and $x\le c(a)y$). Since $M_2-M_1 \le N$, the RHS is
\[
\lesssim N\, \bigl( T^{1/2}R^{-\delta} + R^{-\frac1{K} + \frac{k+1-\beta}{2K-2}} \bigr)\,.
\]
with some $\delta>0$. Since $K\ge 2$, we have
\[
k+1-\beta < 1 \le 2 - \frac2{K} = \frac{2K-2}{K}\,.
\]
Therefore,
\[
\max_{\substack{M_2-M_1\le N \\
|M_1|, |M_2|\le N}} \bigl| S_T (M_1, M_2) \bigr|
\lesssim_{\beta, \delta}\, T^{1/2} R^{1/2-\delta/2}\,,
\]
and, by Lemma~\ref{lemma5}, $X_R \ge c(m)\sqrt R$, provided that $R\ge R_0(m, \beta)$.
\hfill $\Box$

\section{Proof of Theorem~2 ($\beta=\frac32$)}
\label{sec6}

In~\cite{CL}, Chen and Littlewood showed that the zeroes of the function
$F_\xi$ with $\xi (n) = e(n^\beta)$, $1<\beta<\frac32$, are asymptotically 
very close to a sequence of points that are regularly distributed on the
spiral given in polar coordinates by $\theta = -\pi + C(\beta)r^{\beta-1} $.
Their analysis yields
that this $\xi$ is an $L$-sequence.
In fact, they gave a detailed proof for another sequence $\xi (n) = e(n (\log n)^{\beta})$
with $\beta>1$, and mention that their arguments work with minor changes in the case we
consider here. Apparently, it is an intriguing open question which part of their analysis can be extended to the case
$\frac32 \le \beta \le 2$ (or, even to $\beta=\frac32$). Nevertheless, as we will show in
this section, a certain combination of their method with our techniques
is strong enough to show that the sequence $\xi (n) = e(n^{3/2})$ is an $L$-sequence.

\medskip
Everywhere in this part,
\[
W_R(\theta) = \sum_{|n|\le N} e\bigl( (n+R)^{3/2} + (n+R)\theta \bigr)\, e^{-n^2/(2R)}
\]
with $N = R^{1/2}\log R + O(1)$; this differs from our definition in Section 4.1 by a unimodular factor $e(R\theta)$.

\subsection{}
Here,
we give an asymptotic estimate of $W_R$, which will yield
Theorem~\ref{thm:beta} in the case $\beta=\frac32$.

\begin{lemma}\label{lemma:main3/2}
For $R\to\infty$,
\begin{multline*}
W_R(\theta) = \frac{2e(1/8+MR)R^{1/4}}{\sqrt{3}} \sum_{|m|\le \frac12\log R}\,
e\bigl( mR -\tfrac4{27}(M + m - \theta)^3 \bigr)\,  e^{-\frac89 (m-\theta)^2} \\+
O\bigl((\log R)^3\bigr)
\end{multline*}
with $M=\frac32 R^{1/2}$, uniformly in $\theta$.
\end{lemma}
It is worth mentioning that, in the case $1<\beta<\frac32$ considered by Chen and Littlewood,
at most two terms contribute to the corresponding sum on the RHS. This was crucial for
finding the asymptotic locations of zeroes of $F$.

\medskip We split the proof of Lemma~\ref{lemma:main3/2} into several parts.

\subsubsection{}
Take $\chi\in C^\infty_0 [0, +\infty)$ with $\chi\ge 0$,
\[
\chi(t) =
\begin{cases} 1, \quad 0\le x\le N,\\
0, \quad x\ge N+1,
\end{cases}
\]
and set $\chi(z)=\chi(|z|)$, and
\[
u(t) = \chi(t-R) e(t^{3/2}+t\theta) e^{-(t-R)^2/(2R)},\qquad t\in\mathbb R.
\]
Then
\[
W_R(\theta) = \sum_{n\in\bZ} u (n) = \sum_{m\in \bZ} \widehat{u}(m) \qquad
({\rm the\ Poisson\ summation}),
\]
where
\begin{align*}
\widehat u(m) &=
\int_{\mathbb R}u(t)e(-mt)\, \rd t \\
&= e(mR) \int_{\mathbb R}\chi(t)e((t+R)^{3/2}-(m-\theta)(t+R) 
)e^{-t^2/(2R)}\, \rd t \\
&= e(mR) \int_{\mathbb R}\chi(t)e(\psi_m(t))e^{-t^2/(2R)}\, \rd t\,,
\end{align*}
where $ \psi_m(t) = (t+R)^{3/2}-\mu(t+R) $ is ``a phase function'',
and $\mu = m-\theta$ is ``a distorted $m$''. Put
\[
I_m = \int_{\mathbb R}\chi(t)e(\psi_m(t))e^{-t^2/(2R)}\, \rd t\,.
\]

Estimating the integrals $I_m$,
we set $M=\tfrac32R^{1/2}$ and consider separately three cases:
$|m-M|> \log R$, $\frac12 \log R < |m-M| \le \log R$, and $|m-M|\le \frac12 \log R$.
In what follows, we extend $\psi_m$ to an analytic function in $\{z:{\rm Re\,} z>-R\}$ and 
use the Taylor approximation of $\psi_m$ 
in the disk $\{z: |z|\le 10N\}$:
\begin{multline} \label{eq:Taylor-psi}
\psi_m (z) = - \tfrac12 R^{3/2} -\sigma R - \sigma z + \tfrac38 z^2 R^{-1/2} - \tfrac1{16} z^3 R^{-3/2}  \\
+ O\bigl( R^{-1/2} (\log R)^4  \bigr),
\end{multline}
where $\sigma = \mu - M$.

\subsubsection{}
We start with the case $|m-M|>\log R$. Then the derivative of the phase $\psi_m$ is large on the
support of $\chi$, see \eqref{x1} below. We show that for $R\ge R_0$,
\begin{equation}\label{eq:integral-1}
|I_m| \le \frac{e^{-c (\log R)^2}}{(m-M)^2}\,.
\end{equation}

Integrating twice by parts we obtain
\begin{multline*}
I_m = \frac{1}{(2\pi\im)^2}\,
\int_{\mathbb R}
\Bigl[\frac{1}{\psi_m'}\Bigl(\frac{\chi e^{-t^2/(2R)}}{\psi_m'}\Bigr)'\Bigr]'(t)e(\psi_m(t))\, \rd t\\
=\frac{1}{(2\pi \im)^2}\, \int_{\mathbb R} \frac{\lambda(t)}{\psi_m'^2(t)} e(\psi_m(t))e^{-t^2/(2R)}\, \rd t,
\end{multline*}
where
$$
\lambda=\chi''-\frac{3\chi'\psi_m''}{\psi_m'}
- \frac{2\chi' t}{R} - \frac{\chi}{R} + \frac{\chi t^2}{R^2} +
\frac{3\chi \psi_m'' t}{\psi_m' R} - \frac{\chi \psi_m'''}{\psi_m'}
+3\Bigl(\frac{\psi_m''}{\psi_m'}\Bigr)^2\chi.
$$
For $|z| \le N+1$ and $R>R_0$, we have
\begin{multline}
|\psi_m'(z)| = \Bigl| \tfrac32 (R+z)^{1/2} - \tfrac32 R^{1/2} - \sigma \Bigr| \\
\ge |\sigma| - \tfrac32 R^{1/2} \Bigl[ \sqrt{1+\frac{N+1}R} - 1\Bigr]
\ge |\sigma| - (\tfrac34 + o(1)) \log R \ge \tfrac15 |\sigma|\,.\label{x1}
\end{multline}
Since the functions $\psi_m''$, $\psi_m'''$ are bounded on
the disk $\{z\colon |z| \le N+1\}$,
we conclude that $\lambda$ is bounded  
on the same disk. 

Next, we set
\begin{align*}
H(z) &= 2\pi \im \psi_m(z) - \frac{z^2}{2R}\\
&=2\pi \im (R+z)^{3/2}-2\pi \im \Bigl(\frac 32 R^{1/2}+\sigma\Bigr)(z+R) - \frac{z^2}{2R}.
\end{align*}
Then 
\[
I_m = -\frac1{4\pi^2}\, \int_{\bR} \frac{\lambda(t)}{\psi_m'^2(t)}\, e^{H(t)}\, \rd t\,.
\]
Using the Taylor expansion~\eqref{eq:Taylor-psi}, we get 
\begin{multline}
|\exp H(x +\im y)|  
\\
\le C\exp\Bigl(2\pi \sigma y-\frac {3\pi}2 xy R^{-1/2}+\frac{3\pi}{8}x^2yR^{-3/2}-\frac{\pi}{8}y^3R^{-3/2}-\frac{x^2}{2R}+\frac{y^2}{2R}\Bigr)  
\\
= C\exp\Bigl[\Bigl(2\pi \sigma -\frac {3\pi}2 x R^{-1/2}+o(1)\Bigr)y -\frac{x^2}{2R}+\frac{y^2}{2R}\Bigr],\,\,\, |x+iy|\le 3N.
\label{eq1}
\end{multline}
Now, 
\[
4\pi^2|I_m|
\le\Bigl|\int_{\frac N2\le |x|\le N+1}\Bigr|+\Bigl|\int_{|x|\le N/2}\Bigr|\,.
\]
For $|x|\ge \frac12 N$, $|\exp H(x)| \le C \exp[ -x^2/(2R) ] \le \exp [-c (\log R)^2] $.
Thus, the first integral does not exceed
\[
C N |\sigma|^{-2} e^{-c(\log R)^2} \le \frac{e^{-c_1(\log R)^2}}{(m-M)^2}\,.
\]
In the second integral, instead of integrating over the interval $[-\frac12 N, \frac12 N]$,
we integrate over the contour $\Gamma_\sigma$ as on Figure 1.  

\begin{center}
\begin{figure}[ht]
\begin{picture}(390,170)
\thicklines
\put(20,120){\vector(1,0){280}}
\put(330,118){if $\sigma<0$}
\put(164,158){$\Gamma_\sigma$}
\put(261,107){$N/2$}
\put(34,107){$-N/2$}
\put(50,120){\circle*{3}}
\put(160,120){\circle*{3}}
\put(270,120){\circle*{3}}
\put(161,107){$0$}
\put(115,145){\vector(0,1){5}}
\put(115,125){\vector(0,-1){5}}
\put(115,120){\line(0,1){30}}
\put(86,133){$R^{1/2}$}
\put(155,150){\vector(1,0){10}}
\put(50,120){\vector(0,1){20}}
\put(270,150){\vector(0,-1){20}}
\put(50,120){\line(0,1){30}}
\put(270,120){\line(0,1){30}}
\put(50,150){\line(1,0){220}}

\thicklines
\put(20,50){\vector(1,0){280}}
\put(330,48){if $\sigma>0$}
\put(164,6){$\Gamma_\sigma$}
\put(86,31){$R^{1/2}$}
\put(261,57){$N/2$}
\put(34,57){$-N/2$}
\put(50,50){\circle*{3}}
\put(160,50){\circle*{3}}
\put(270,50){\circle*{3}}
\put(161,57){$0$}
\put(155,20){\vector(1,0){10}}
\put(115,45){\vector(0,1){5}}
\put(115,25){\vector(0,-1){5}}
\put(50,50){\line(0,-1){30}}
\put(270,50){\line(0,-1){30}}
\put(50,20){\line(1,0){220}}
\put(115,50){\line(0,-1){30}}
\put(50,50){\vector(0,-1){20}}
\put(270,20){\vector(0,1){20}}

\end{picture}
\caption*{Figure 1}
\end{figure}
\end{center}

Estimate~\eqref{eq1} shows that
$$
|e^{H(x+\im y)}|\le
\begin{cases}
C\, e^{-cR^{1/2}\log R}, \ &z\in\Gamma_\sigma,\, |y|=R^{1/2},\\
C\, e^{-c(\log R)^2}, \   &z\in\Gamma_\sigma,\,  |x|=N/2.
\end{cases}
$$
Therefore, the integral over the contour $\Gamma_\sigma$ is also bounded by
$$
(m-M)^{-2}\, e^{-c(\log R)^2},
$$
and estimate~\eqref{eq:integral-1} follows.

\subsubsection{}
Now, $\frac12\log R< |m-M|\le\log R$. This case is similar to the previous one but is somewhat shorter
since there is no need to integrate by parts (instead of \eqref{eq:integral-1} we check a simpler estimate \eqref{x2}). We again split the integral into two parts:
\begin{multline*}
|I_m|
= \Bigl|\int_{\mathbb R}\chi(t)e(\psi_m(t))e^{-t^2/(2R)}\, \rd t\Bigr|
\le \Bigl|\int_{\frac N2\le |x|\le N+1}\Bigr|+\Bigl|\int_{|x|\le N/2}\Bigr|\\ \le C e^{-c(\log R)^2}+
\Bigl| \int_{\Gamma_\sigma} e^{H(z)}\, \rd z \Bigr|,
\end{multline*}
and, arguing as above, we obtain
\begin{equation}
| I_m| \le e^{-c (\log R)^2}\,.
\label{x2}
\end{equation}

\subsubsection{}
At last, we deal with $I_m$ such that $|m-M|\le \frac12\log R$.
This case requires a saddle point approximation.
Set
\[
z_0 = \frac43 \sigma R^{1/2}, \qquad
A_0 = -\frac89 \sigma^2 - \frac{8\pi\im}{27}\mu^3.
\]
Then, using the Taylor approximations~\eqref{eq:Taylor-psi}, we get
\begin{align*}
H(z_0)&=A_0 + O\bigl( R^{-1/2}(\log R)^4 \bigr),\\
H'(z_0)&=O\bigl(R^{-1/2} (\log R)^2 \bigr),\\
H''(z)&=\frac{3\pi\im}{2}R^{-1/2}+O(R^{-1}\log R), \qquad |z|<5N\,.
\end{align*}

Now,
$$
I_m = \int_{|x|\le N} + \int_{N\le |x|\le N+1}   =
\int_{\Lambda_\sigma}e^{H(z)}\, \rd z   + O\bigl( e^{-c(\log R)^2} \bigr)\,,
$$
where the Fresnel-type contour $\Lambda_\sigma$ is as on Figure~2.

\begin{center}
\begin{figure}[ht]
\begin{picture}(390,300)
\thicklines
\put(20,180){\vector(1,0){360}}
\put(200,180){\circle*{3}}
\put(230,180){\circle*{3}}
\put(70,180){\circle*{3}}
\put(330,180){\circle*{3}}
\put(203,187){$0$}
\put(233,170){$z_0$}
\put(195,145){\vector(1,1){6}}
\put(204,134){$\Lambda_\sigma$}
\put(56,188){$-N$}
\put(327,166){$N$}
\put(254,189){$\pi/4$}
\put(70,20){\line(1,1){260}}
\put(70,20){\line(0,1){160}}
\put(330,180){\line(0,1){100}}
\put(230,180){\arc[0,45]{19}}
\put(70,150){\vector(0,-1){50}}
\put(330,250){\vector(0,-1){30}}
\end{picture}
\caption*{Figure 2}
\end{figure}
\end{center}

Let $\Lambda_\sigma^0$ be the vertical part of $\Lambda_\sigma$, and $\Lambda_\sigma^1$ be the rest.
Then by estimate~\eqref{eq1} we have
$$
\Bigl| \int_{\Lambda^0_\sigma}e^{H(z)}\, \rd z\Bigr|\le e^{-c(\log R)^2}.
$$
Hence,
\[
I_m = \int_{\Lambda^1_\sigma}e^{H(z)}\, \rd z + O\bigl( e^{-c(\log R)^2} \bigr)\,.
\]
Furthermore,
$$
\int_{\Lambda^1_\sigma}e^{H(z)}\rd z = e^{\im\pi/4}\,
\int_{-R^{1/2}(NR^{-1/2}+4\sigma/3)}^{R^{1/2}(NR^{-1/2}-4\sigma/3)}e^{H(te^{\im\pi/4}+z_0)}\, \rd t\,.
$$
For $|t|\le 2N$ we have
\begin{multline*}
H(te^{\im\pi/4}+z_0)\\
= A_0+O\Bigl(\frac{(\log R)^4}{R^{1/2}}\Bigr)
+O\Bigl(t\frac{(\log R)^2}{R^{1/2}}\Bigr)-\frac{3\pi}{4}t^2R^{-1/2}+O\Bigl(t^2\frac{\log R}{R}\Bigr),
\end{multline*}
and hence,
\begin{align*}
\int_{\Lambda^1_\sigma}\, &e^{H(z)}\, \rd z \\
&= e^{\im\pi/4+A_0}\, R^{1/4}\!\!\!\!\!\int\limits_{R^{1/4}(-\log R-4\sigma/3+o(1))}^{R^{1/4}(\log R-4\sigma/3+o(1))}\!\!\!\!
\exp\Bigl[-\frac{3\pi}{4}t^2\\
&\qquad\qquad\qquad\qquad\qquad + O\Bigl(\frac{(\log R)^4}{R^{1/2}} + t\frac{(\log R)^2}{R^{1/4}}+t^2\frac{\log R}{R^{1/2}}\Bigr)\Bigr]\, \rd t \\
&= e^{\im\pi/4+A_0}\, R^{1/4}
\Bigl(\int_{\mathbb R}\exp\bigl[-\frac{3\pi}{4}t^2\bigr]dt+O\Bigl(\frac{(\log R)^3}{R^{1/4}}\Bigr)\Bigr)\\
&= \frac{2}{\sqrt 3}R^{1/4}e\bigl(\tfrac 18-\tfrac{4}{27}\mu^3\bigr)e^{-\frac 89 \sigma^2}
+O\bigl((\log R)^3\bigr).
\end{align*}
Finally, for $|m-M|\le \frac12 \log R$, we get
\begin{align*}
I_m &= \frac{2}{\sqrt 3}R^{1/4}e\bigl(\tfrac 18-\tfrac{4}{27}\mu^3\bigr)
e^{-\frac 89 \sigma^2}
+ O\bigl((\log R)^3\bigr) + O\bigl( e^{-c(\log R)^2} \bigr) \\
&= \frac{2}{\sqrt 3}R^{1/4}e\bigl(\tfrac 18-\tfrac{4}{27}(m-\theta)^3\bigr)e^{-\frac 89
(m-M-\theta)^2}
+ O\bigl((\log R)^3\bigr).
\end{align*}

\subsubsection{}
Thus,
\begin{multline*}
W_R(\theta) =\sum_{m\in\mathbb Z}\widehat{u}(m+M)\\= \Bigl[ \sum_{|m|\le \frac12 \log R} + \sum_{\frac12 \log R < m \le \log R} + \sum_{|m|> \log R} \Bigl]
e ((M+m)R) I_{m+M} \\
= \frac{2e(1/8 + MR)R^{1/4}}{\sqrt 3}
\sum_{|m|\le \frac12 \log R} e\bigl(mR-\tfrac{4}{27}(M+m-\theta)^3\bigr)e^{-\frac 89 (m-\theta)^2}\\
+ O\bigl((\log R)^3\bigr),
\end{multline*}
proving Lemma~\ref{lemma:main3/2}. \hfill $\Box$

\subsection{}
At last, we are able to prove Theorem~\ref{thm:beta} for $\beta=\frac32$.
Consider the shifts $W_R(\theta+t)$ with $0\le t\le \frac98 M^{-1}$. We have
\begin{multline*}
W_R(\theta+t) \\= \frac{2e(1/8 + M R)R^{1/4}}{\sqrt 3}  
\sum_{|m|\le \frac12 \log R} e\bigl(mR - \tfrac{4}{27}(M+m-\theta-t)^3\bigr)e^{-\frac 89 (m-\theta-t)^2}\\
+  O\bigl((\log R)^3\bigr).
\end{multline*}
Furthermore, since $|t|=O\bigl( M^{-1} \bigr)$ with $M=\frac32 R^{1/2}$, we have
\begin{multline*}
e\bigl(-\tfrac{4}{27}(M+m-\theta-t)^3\bigr)\,e^{-\frac 89 (m-\theta-t)^2} \\
 = e\bigl(-\tfrac{4}{27}(M+m-\theta)^3 + \tfrac49 (M^2-2M\theta)t + \tfrac89 Mmt\bigr)\, e^{-\frac 89 (m-\theta)^2}\\
+  O\Bigl(\frac{(\log R)^2}{R^{1/2}}\Bigr).
\end{multline*}
Therefore,
\begin{multline*}
W_R(\theta+t) \\= K R^{1/4}
\sum_{|m|\le \frac12 \log R} e\bigl(\tfrac89 M t m + mR -\tfrac{4}{27}(M+m-\theta)^3\bigr)e^{-\frac 89 (m-\theta)^2}\\
+  O\bigl((\log R)^3\bigr)
\end{multline*}
with
\[
K = K(M, \theta, t) = \frac{2e\bigl( \tfrac18 + MR + \tfrac49 (M^2-2M\theta)t\bigr)}{\sqrt 3}\,.
\]

Now, notice that the sum on the RHS is a Fourier series in the variable $\frac 89 Mt$. Hence, by Parseval's theorem,
there exists $t\in [0, \frac98 M^{-1}]$ so that
\[
| W_R(\theta+t) | \ge \frac{2 R^{1/4}}{\sqrt3 }\,
\Bigl(\sum_{|m|\le \frac 12\log R}e^{-\frac {16}9 (m-\theta)^2}\Bigr)^{1/2} -   O\bigl((\log R)^3\bigr)
\ge CR^{1/4}
\]
with a positive numerical constant $C$. Applying Lemma~\ref{lemma4},
we finish off the proof. \hfill $\Box$

\section{Wide-sense stationary sequences}\label{sect:wide-stat}

Here, we prove several simple  
lemmas pertaining to the case when $\xi\colon \bZ_+ \to \bC$
is a {\em wide-sense stationary sequence}, that is, $\bE |\xi (n)|^2 < \infty$ for every $n$,
and $\bE \xi (n)$ and $\bE \bigl\{ \xi (n) \overline{\xi (n+m)} \bigr\} $ do not depend on $n$.
We also always assume that $\xi$ is not the zero sequence.
By $\rho$ we denote {\em the spectral measure} of such a sequence $\xi$. That is,
$\rho$ is a finite non-negative measure on the unit circle $\bT$ such that
\[
\bE \bigl[ \xi (n_1) \overline{\xi (n_2)} \bigr] = \widehat\rho (n_2-n_1)\,,
\]
and by $\sigma (\xi)$ we denote the spectrum of $\xi$, that is, the closed support of the spectral
measure $\rho$. In what follows, by $\sigma^*$ we always denote the reflection of the spectrum
$\sigma$ in the real axis.

Observe that if $\xi$ is a wide-sense stationary sequence then, almost surely,
$F_\xi$ is an entire function of exponential type at most one. Indeed, for every $\epsilon>0$,
\[
\bP \bigl\{ |\xi (n)| > (1+\epsilon)^n \bigr\} \le (1+\epsilon)^{-2n}\, \bE |\xi (n)|^2\,,
\]
whence, by the Borel--Cantelli lemma,
\[
\limsup_{n\to\infty} |\xi (n)|^{1/n} \le 1, \qquad {\rm almost\ surely},
\]
which is equivalent to the inequality
$ |F_\xi (z)| \le C(\epsilon) e^{(1+\epsilon)|z|} $ valid for every $z\in\bC$ and every $\epsilon>0$.

\subsection{}
First, we compute the variance of $F_\xi$ in terms of the spectral measure $\rho$.

\begin{lemma}\label{lemma8.1}
Suppose $\xi$ is a wide-sense stationary sequence. Then
\begin{equation}\label{eq:8.1}
\bE \bigl| F_\xi (re^{\im\theta}) \bigr|^2 = \int_{-\pi}^\pi e^{2r\cos (\theta+t)}\, \rd\rho (t)\,,
\end{equation}
and
\begin{equation}\label{eq:8.1-asympt}
\log \bE \bigl| F_\xi (re^{\im\theta}) \bigr|^2 =
2r h_{\sigma^*}(\theta) + o(r)\,, \qquad r\to\infty\,.
\end{equation}
\end{lemma}

\begin{proof}
We have
\begin{align*}
\bE \bigl| F_\xi (re^{\im\theta}) \bigr|^2
&=
\sum_{n_1, n_2 \ge 0} \bE \bigl[ \xi (n_1) \overline{\xi (n_2)} \bigr]\, e^{\im (n_1-n_2)\theta}\,
\frac{r^{n_1+n_2}}{n_1! n_2!} \\
&= \sum_{n_1, n_2 \ge 0} \Bigl[ \int_{-\pi}^\pi e^{- \im (n_2-n_1)t}\, \rd\rho (t) \Bigr]\, e^{\im (n_1-n_2)\theta}\,
\frac{r^{n_1+n_2}}{n_1! n_2!} \\
&= \int_{-\pi}^\pi
\Bigl[ \sum_{n_1, n_2 \ge 0} e^{\im n_1(\theta+t)}\, \frac{r^{n_1}}{n_1!} \cdot
e^{- \im n_2 (\theta+t)}\, \frac{r^{n_2}}{n_2!} \Bigr]\, \rd\rho (t) \\
&= \int_{-\pi}^\pi e^{r\bigl[ e^{\im (\theta+t)} + e^{- \im (\theta+t)} \bigr]}\, \rd\rho (t) \\
&= \int_{-\pi}^\pi e^{2r \cos  (\theta+t)}\, \rd\rho (t)\,,
\end{align*}
proving~\eqref{eq:8.1}. Now, recalling the  definition of
the supporting function
\[
h_{\sigma^*} (\theta) = \, \max_{t\in {\rm spt\, }(\rho)}\, \cos (\theta + t)\,,
\]
we readily get asymptotics~\eqref{eq:8.1-asympt}.
\end{proof}

\subsection{}
As a straightforward consequence of the previous lemma, we get
\begin{lemma}\label{lemma8.2}
Suppose $\xi$ is a wide-sense stationary sequence. Then, almost surely,
\[
h^{F_\xi} (\theta) \le h_{\sigma^*} (\theta)\,, \qquad \theta\in [-\pi, \pi]\,.
\]
\end{lemma}
In other words, the indicator diagram $I^{F_\xi}$
of $F_\xi$ is contained in the closed convex hull
of the spectrum $\sigma (\xi)$ reflected in the real axis.

\begin{proof}
Using~\eqref{eq:8.1-asympt} and Chebyshev's inequality, we see that, for every $\epsilon>0$,
\begin{multline*}
\bP \bigl\{ \log |F_\xi (re^{\im\theta})| > ( h_{\sigma^*}(\theta) + \epsilon) r \bigr\}
= \bP \bigl\{ |F_\xi (re^{\im\theta})|^2 > e^{2 ( h_{\sigma^*}(\theta) + \epsilon) r} \bigr\} \\
\le \bE \bigl\{ |F_\xi (re^{\im\theta})|^2 \bigr\}\, e^{- 2 ( h_{\sigma^*}(\theta) + \epsilon) r}
= e^{-2\epsilon r + o(r)}\,, \qquad r\to\infty\,.
\end{multline*}
Whence, by the Borel--Cantelli lemma, for every $\kappa>0$ and every $\theta\in [-\pi, \pi]$,
\[
\limsup_{n\to\infty} \frac{\log |F_\xi (\kappa n e^{\im\theta})|}{\kappa n} \le h_{\sigma^*} (\theta),
\qquad {\rm almost\ surely}.
\]
Since the exponential type of the entire function $F_\xi$ does not exceed $1$, for any $\kappa<\pi$,
we have
\[
\limsup_{r\to\infty} \frac{\log |F_\xi (r e^{\im\theta})|}{r}
= \limsup_{n\to\infty} \frac{\log |F_\xi (\kappa n e^{\im\theta})|}{\kappa n}\,.
\]
This is a special instance of a classical result that goes back to P\'olya and to Vl.~Bernstein.
For a simple proof of this result see, for instance,~\cite[Theorem~1.3.5]{Bieberbach}.
Therefore, given $\theta\in [-\pi, \pi]$, almost surely, we have
$ h^{F_\xi} (\theta) \le h_{\sigma^*} (\theta) $. Since both functions in this inequality are
continuous on $[-\pi, \pi]$, we immediately conclude that, almost surely, the inequality holds
for all $\theta\in [-\pi, \pi]$.
\end{proof}
We will be using Lemmas~\ref{lemma8.1} and~\ref{lemma8.2} in the Gaussian case (Theorem~\ref{thm:stat-gauss}).

\subsection{}
The next lemma is needed for the mixing case (Theorem~\ref{thm:stat-corr}).
As above, we use the notation
\[
W_R(\theta) = \sum_{|n|\le N} \xi (n+R)  e(n\theta)\, e^{-\frac{n^2}{2R}}\,,
\qquad N = R^{1/2} \log R +O(1),
\]
fix a non-negative even function $g\in C^2_0[-\tfrac12, \tfrac12]$ with
$\int g(\theta)\,\rd\theta =1$, and set 
\begin{multline*}
X_R = \int_{a-\frac1{2m}}^{a+\frac1{2m}} \bigl| W_R(\theta) \bigr|^2
g(m(\theta-a))\, \rd\theta \\
= \frac1{m} \, \sum_{|n_1|, |n_2| \le N}\,
\xi(n_1+R) \overline{\xi (n_2+R)}
e\bigl((n_1-n_2)a\bigr) e^{-(n_1^2+ n_2^2)/(2R)}\, \widehat{g}\bigl( \frac{n_2-n_1}m \bigr)\,.
\end{multline*}

\begin{lemma}\label{lemma8.3a}
Suppose $\xi$ is a wide-sense stationary sequence whose spectral measure $\rho$ has
no gaps in its support. Then for every $a\in [0, 1]$ and every $m\in \bN$, there
exists a positive limit
\begin{equation}\label{eq:8.3}
\lim_{R\to\infty} R^{-1/2}\, \bE X_R = c(a, m) >0\,.
\end{equation}
\end{lemma}

\begin{proof}
We have
\[
\bE X_R = \frac1{m} \, \sum_{|n_1|, |n_2| \le N}\,
\widehat{\rho} (n_2-n_1)\, \widehat{g}\, \bigl( \frac{n_2-n_1}{m} \bigr)\, e((n_1-n_2)a)\,
e^{-(n_1^2+n_2^2)/{2R}}\,.
\]
Put $k=n_2-n_1$, $\ell = n_2+n_1$. Then
\[
|k| \le 2N, \quad |\ell|\le 2N-k, \quad \ell \equiv k \,{\rm mod}\, 2,
\]
and $ n_1^2 + n_2^2  = \frac12 (k^2 +\ell^2)$. Hence,
\[
\bE X_R = \frac 1m\sum_{|k|\le 2N} \widehat{\rho}(k)\, \widehat{g}\bigl( \frac{k}{m}\bigr)\,
e(-ka)\, e^{-k^2/(4R)}\, \sum_{\substack{|\ell|\le 2N-k\\ \ell\equiv k \,{\rm mod}\, 2}} e^{-\ell^2/(4R)}\,.
\]
Because of the cut-off $e^{-k^2/(4R)}$, we discard the sum over $N\le |k|\le 2N$ (recall that
$N = R^{1/2} \log R+O(1)$) and consider only the range $|k|\le N$. Then the inner ``$\ell$-sum'' equals
$\sqrt{\pi R} + O(R^{-1/2})$, and we get
\[
\bE X_R = \frac{\sqrt{\pi R}}{m}\, \sum_{|k|\le N}
\widehat{\rho}(k)\, \widehat{g}\bigl( \frac{k}{m}\bigr)\,
e(-ka)\, e^{-k^2/(4R)} + O( \log R).
\]
Since $\widehat{g}\in l^1(\mathbb Z)$, by the dominated convergence on $\mathbb Z$ we have 
\[
\lim_{R\to\infty} R^{-1/2} \bE X_R =\frac{\sqrt{\pi}}{m}\, \sum_{k\in\bZ}
\widehat{\rho}(k)\, \widehat{g}\bigl( \frac{k}{m}\bigr)\, e(-ka)\,.
\]
The sum on the RHS is the density of the convolution  $\rho \ast g_m$ at the point $-a$, where
\[
g_m(\theta) =
\begin{cases}
m g(m\theta), & |\theta|\le 1/(2m) \\
0, & {\rm otherwise}.
\end{cases}
\]
Since the support of $\rho$ is the whole circle $\bT$, and the function $g$ is non-negative, this value is positive. This proves the lemma.
\end{proof}

\section{Proof of Theorem~6}

\subsection{}
First, we assume that $\xi\colon \bZ \to \bZ$ is an integer-valued stationary sequence with the spectral
measure $\rho$.
Let $K_\xi$ be the convex hull of $\operatorname{spt}(\rho)$, and let $K_\xi^*$ be its reflection in the real axis. Suppose that $\operatorname{spt}(\rho)\not=\bT$, that is $K_\xi^*\not=\overline{\mathbb D}$.
By P\'olya's theorem (see \cite[Theorem~33, Chapter~I]{Levin} or~\cite[Theorem~1.1.5]{Bieberbach}), the series
\[
f_\xi (w) = \sum_{n\ge 0} \frac{\xi (n)}{w^{n+1}}
\]
is analytic on $\widehat{\bC}\setminus K_\xi^*$. Since $\xi$ attains only integer values, another theorem of P\'olya~\cite[Theorem~6.2.1]{Bieberbach}
yields that for every fixed $\xi$, the function $f_\xi$ is rational with poles at roots of $1$,
$f_\xi=P/Q$,
with mutually prime $P,Q\in \mathbb Z[w]$ and monic $Q$.

Next, we use simple algebra.
Noting that $P$ is a product of irreducible polynomials, and recalling that if a polynomial
is irreducible in $\bZ[w]$ then it is also irreducible in $\bQ[w]$ (``Gauss lemma''), and that
two different irreducible polynomials in $\bQ[w]$ are mutually prime, we conclude that
$P$ has no common zeroes with $Q$.

Since any polynomial in $\bZ[w]$ is a product of irreducible ones, and since cyclotomic
polynomials
\[
\Phi_{n}(w)=\prod_{\gcd(k,n)=1}(w-e(k/n))
\]
belong to $\bZ[w]$ and are irreducible therein, we see that
\[
Q=\prod_{1\le k\le u}\Phi_{n(k)}\,.
\]
Since $f_\xi$ is analytic on a fixed arc of the unit circle, we obtain that $n(k)\le M$ for some $M$ independent of $\xi$.
Thus, the set of poles of $f_\xi$ is contained in $\{w\colon w^N=1\}$ for some $N\ge 1$ independent of $\xi$.

Furthermore, since $\bE |\xi (n)|^2$ is finite (and does not depend on $n$), applying Chebyshev's inequality
and the Borel-Cantelli Lemma, we see that, for any $\lambda>\frac12$, almost surely, $|\xi (n)| = o(n^\lambda)$,
whence,
\[
\max_{|w|=r} |f_\xi (w)| = o((r-1)^{-2}), \qquad r\downarrow 1\,.
\]
Therefore, all poles of $f_\xi$ are simple. Thus, $f_\xi$ can be written in the form
$ f_\xi (w) = (w^N-1)^{-1} S(w) $, where $S$ is a polynomial (depending on $\xi$)
and $N\in\bN$ does not depend on $\xi$. 
Hence, the coefficients $\xi (n)$ of $f_\xi$ are eventually
periodic with period $N$.

Since the sequence $\xi$ is stationary, we conclude that it is periodic with period $N$. 
Indeed, given $M<\infty$ we consider the bounded sequence $\xi_M$ given by
$$
\xi_M(n)=\begin{cases}M,\quad \xi(n)>M,\\ \xi(n),\quad -M\le \xi_M(n)\le M,\\-M,\quad \xi(n)<-M.\end{cases}
$$
The $\xi_M$ is also stationary, and the values $\bE(\xi_M(k+N)-\xi_M(k))^2$ do not depend on $k$. 
On the other hand, almost surely, the sequence $\xi_M(n)$ is eventually periodic with period $N$, and 
by the bounded convergence theorem, the values $\bE(\xi_M(k+N)-\xi_M(k))^2$ converge to $0$ for $k\to\infty$.
Hence they are equal to $0$, and the sequences $\xi_M(n)$ are periodic with period $N$ for every $M$. 
Thus, $\xi$ is periodic with period $N$. 
\hfill$\Box$

\medskip
Note that we used stationarity of $\xi$ only on the last step of the proof. The rest is
valid for wide-stationary integer-valued sequences. Also note that this last step can be made for 
wide-stationary sequences $\xi$ satisfying the condition $\sup_n\bE |\xi(n)|^\kappa<\infty$ for some 
$\kappa>2$. Hence, the first statement in Theorem~\ref{thm:6} is valid for wide-stationary sequences satisfying this moment condition. 

\subsection{}
To prove the second part of Theorem~\ref{thm:6}, we use a result of Hausdorff~\cite[Theorem~4.2.4]{Bieberbach}.
It says that {\em if the set $A$ is uniformly discrete then there exist at most countably
many sequences $\xi$ such that the series $f_\xi (w)$ can be analytically continued
through an arc in $\bT$}.

Let $\mu$ be a translation invariant probability measure in the space of sequences
$A^{\bZ}$ corresponding to the stationary sequence $\xi$. Suppose that there exists
a lacuna in the support of the spectral measure $\rho$. Then, as above,
by Lemma~\ref{lemma8.2} combined with  P\'olya's theorem, almost surely,
the function $f_\xi$ has an analytic continuation through an arc in $\bT$; and by
the theorem of Hausdorff, the measure $\mu$ has at most countable support.
Since $\mu$ is translation invariant, we conclude that, almost surely, the sequence
$\xi(n)$ is periodic. Since $\mu$ is ergodic, the sequence
$\xi$ is periodic. \hfill $\Box$

\section{Proof of Theorem~3}

\subsection{}
The proof of Theorem~\ref{thm:stat-corr} needs in addition an estimate of
the fourth order correlations:

\begin{lemma}\label{lemma8.3b}
Let $\xi$ be a bounded stationary sequence of random variables,  
and let the maximal correlation coefficient of $\xi$ satisfy
\[
r(t) = O\bigl( (\log t)^{-\kappa}  \bigr)\,, \qquad t\to\infty\,,
\]
with some $\kappa>1$. Then, for every $a\in [0, 1]$ and every $m\in\bN$
\[
\bE \bigl( X_R - \bE X_R \bigr)^2
= O\Bigl( \frac{R}{(\log R)^{\kappa_1}}\Bigr)\,,
\qquad R\to\infty\,,
\]
with some $1<\kappa_1<\kappa$.
\end{lemma}

\begin{proof} We have
\begin{multline*}
m^2\,\bE \bigl( X_R - \bE X_R \bigr)^2 \\
= \bE \Bigl[ \sum_{|n_1|, |n_2|\le N}
\bigl( \xi(n_1+R)\overline{\xi(n_2+R)} -
\bE \bigl\{  \xi(n_1+R)\overline{\xi(n_2+R)}\bigr\} \bigr)\,
\widehat{g}\bigl( \frac{n_2-n_1}{m} \bigr) \times \\
\times e((n_1-n_2)a)\, e^{(n_1^2+n_2^2)/(2R)} \Bigr]^2 \\
= \sum_{|n_1|, ..., |n_4|\le N}\,
C(n_1, n_2, n_3, n_4)\, \widehat{g}\bigl( \frac{n_2-n_1}{m} \bigr)\,
\widehat{g}\bigl( \frac{n_4-n_3}{m} \bigr) \times \\
\times e((n_1-n_2+n_3-n_4)a)\,  e^{(n_1^2 + n_2^2 + n_3^2 + n_4^2)/(2R)}\,,
\end{multline*}
where
\[
C(n_1, n_2, n_3, n_4) =
\bE \bigl\{ \eta (n_1, n_2) \cdot \eta (n_3, n_4) \bigr\}
\]
with
\[
\eta (n_i, n_j) = \xi (n_i+R) \overline{\xi (n_j+R)} -
\bE \bigl\{ \xi (n_i+R) \overline{\xi (n_j+R)}\bigr\}\,.
\]
Let $I$ be the interval with endpoints $n_1$ and $n_2$, 
and let $J$ be the interval with endpoints $n_3$ and $n_4$.
Denoting $t=\dist(I,J)$
%\[
%t = \min\bigl\{ |n_i-n_j|\colon i\in\{1, 2\}, j\in\{3, 4\}\bigr\}\,,
%\]
we estimate $C$ by the maximal correlation coefficient $r(t)$:
\[
\bigl| C(n_1, n_2, n_3, n_4)  \bigr| \le
r(t)\, \sqrt{\bE| \eta (n_1, n_2) |^2  \cdot \bE| \eta (n_3, n_4) |^2}
\le 4r(t)\,\| \xi \|_\infty^4\,.
\]
Therefore,
\begin{multline*}
\bE \bigl( X_R - \bE X_R \bigr)^2
= O(1)\|\xi\|_\infty^4 \times
\\
\times \sum_{|n_1|, ..., |n_4|\le N}\, r(t) \,
\frac{m^2}{1+(n_1-n_2)^2} \, \frac{m^2}{1+(n_3-n_4)^2} \,
e^{-(n_1^2+n_2^2+n_3^2 + n_4^2)/(2R)}\,.
\end{multline*}

To estimate the sum on the RHS, we put
\[
k_1 = n_1-n_2, \ \ell_1=n_1+n_2, \ k_2 = n_3-n_4, \ \ell_2 = n_3 + n_4\,.
\]
Then %$ t = \frac12 \min| \pm k_1 \pm k_2 + (\ell_1-\ell_2) | $,
%where the minimum is taken over all combinations of signs. Hence,
$ t \ge \frac12 \bigl( |\ell_1-\ell_2| - (|k_1|+|k_2|) \bigr) $,
and we need to estimate the sum
\[
\sum_{|k_1|, |k_2|, |\ell_1|, |\ell_2|\le 2N}\
r\bigl( \tfrac12 \bigl( |\ell_1-\ell_2| - (|k_1|+|k_2|) \bigr) \bigr)\,
\frac{e^{- (k_1^2 + k_2^2 +\ell_1^2 + \ell_2^2)/(4R)}}{(1+k_1^2)(1+k_2^2)}\,.
\]
Here and later on, $r(t)=r(\max([t],0))$, where $[t]$ is the maximal integer not exceeding $t$.

We split this sum into two parts. The first one is taken over
$|\ell_1-\ell_2|\le 2(|k_1|+|k_2|)$, while the second one is
taken over $|\ell_1-\ell_2| > 2(|k_1|+|k_2|)$.

The first sum does not exceed
\begin{align*}
&\sum_{k_1, k_2 \ge 0}
\frac{e^{- (k_1^2 + k_2^2)/(4R)}}{(1+k_1^2)\,(1+ k_2^2)}
\cdot O(1+k_1+k_2) \cdot  O(N)  \\
&=O(R^{1/2} \log R)\, \sum_{k_1, k_2\ge 0}\, \frac{1+k_1+k_2}
{(1+k_1^2)\,(1+ k_2^2)}\, e^{-(k_1^2+k_2^2)/(4R)} \\
&=O(R^{1/2} \log R)
\Bigl[ \sum_{k\ge 1} \frac{e^{-k^2/(4R)}}{k} + O(1) \Bigr] = O\bigl( R^{1/2} (\log R)^2 \bigr)\,,
\end{align*}
while the second sum is bounded by
\[
O(1)\, \sum_{|\ell_1|, |\ell_2| \le 2N}\,
r\bigl( \tfrac14 |\ell_1-\ell_2| \bigr) e^{-(\ell_1^2 +\ell_2^2)/(4R)}
= O(\sqrt{R})\, \sum_{\ell>1}\, r\bigl( \tfrac14 \ell \bigr) e^{-\ell^2/(8R)}\,.
\]
Recall that $r(t)=O\Bigl( \frac1{\log^\kappa R}\Bigr)$ and let $\kappa=1+2\epsilon$. Then
\begin{multline*}
\sum_{\ell>1}\, r\bigl( \tfrac14 \ell \bigr) e^{-\ell^2/(8R)}
\le \, \sum_{1\le\ell\le\sqrt{R}\log^\epsilon R}\, r\bigl( \tfrac14 \ell \bigr)
+ \, \sum_{\ell>\sqrt{R}\log^\epsilon R}\, e^{-\ell^2/(8R)} \\
= O(1)\,
\Bigl[ \frac{\sqrt{R}\, \log^\epsilon R}{\log^{1+2\epsilon} R} + \sqrt{R}\, e^{-c \log^{2\epsilon}R} \Bigr]
=
O(1)\, \frac{\sqrt{R}}{\log^{1+\epsilon}R}\,.
\end{multline*}
This completes the proof of Lemma~\ref{lemma8.3b}.
\end{proof}

\subsection{}
Now, the proof of Theorem~\ref{thm:stat-corr} is straightforward.
Since the maximal correlation coefficient $r(m)$ decays to zero as $m\to\infty$, the
bounded stationary sequence $\xi$ is {\em linearly regular}. That is,
$ \displaystyle \bigcap_m L^2_{(-\infty, m]} = \{ 0 \} $, where $ L^2_{(-\infty, m]} $ is the Hilbert space,
which consists of the random variables measurable with respect to  
the $\sigma$-algebra generated by $\bigl\{\xi (n)\colon -\infty<n\le m \bigr\}$
that have a finite second moment.  Then the spectral measure $\rho$ has a density
$|f|^2$, where $f$ belongs to
the Hardy space $H^2(\bT)$, see~\cite[Chapter~XVII, \S1]{IL}, and therefore, ${\rm spt}(\rho)=\bT$.
Hence, we are in the assumptions of Lemma~\ref{lemma8.3a}.
Fix $a\in[0,1]$, $m\ge 1$.
Then, combining Lemma~\ref{lemma8.3a} with Lemma~\ref{lemma8.3b}, and using Chebyshev's
inequality, we see that,
for some $c = c(a, m)>0$, $\kappa>0$,
\[
\bP \bigl\{ X_R < c\sqrt{R} \bigr\} = O(R^{-1})\, \bE\bigl( X_R - \bE X_R \bigr)^2
= O\bigl( (\log R)^{-\kappa} \bigr)\,.
\]
Then we take any $\delta\in (\kappa^{-1}, 1)$, and put $R_j=e^{j^\delta}$.
This is a thick sequence (i.e., $R_{j+1}/R_j\to 1$), while
\[
\bP \bigl\{ X_{R_j} < c\sqrt{R_j} \bigr\} = 
O\bigl( j^{-\delta\kappa}\bigr)
\]
with $\delta\kappa>1$. Applying the Borel--Cantelli lemma, we get estimate~\eqref{eq4.4}.
Then Lemma~\ref{lemma4.5} does the job. \hfill $\Box$

\section{Proof of Theorem~4}\label{sec10}

Given $z=re^{\im\theta}$, $F_\xi(z)$ is a Gaussian random variable. As before,
$\sigma^*$ is the reflection of the spectrum $\sigma(\xi)$ in the real axis.
By Lemma~\ref{lemma8.1},
\[
\bE |F_\xi (re^{\im\theta})|^2 = e^{2h_{\sigma^*}(\theta)r + o(r)}\,, \qquad r\to\infty\,.
\]
Then, for every $\epsilon>0$, every $r>r_\epsilon$, and every $\theta\in [-\pi, \pi]$, we have
\begin{multline*}
\bP \bigl\{ \log |F_\xi (re^{\im\theta})| < (h_{\sigma^*}(\theta)-\epsilon)r \bigr\} \\
= \bP \bigl\{ |F_\xi (re^{\im\theta})| < e^{-\epsilon r + o(r)} \, \sqrt{\bE |F_\xi (re^{\im\theta})|^2}\, \bigr\}
< e^{-\frac12\epsilon r}
\end{multline*}
(the last inequality is the place where we are using the Gaussianity of $F_\xi$).
Applying this with $R=j$ and using the Borel--Cantelli lemma,
we see that, given $\theta\in [-\pi, \pi]$, we have
\[
\liminf_{j\to\infty} \frac1{j}\, \log |F_\xi (je^{\im\theta})| \ge h_{\sigma^*}(\theta),
\qquad {\rm  almost\ surely}.
\]
By Lemma~\ref{lemma8.2}, $h^{F_\xi} \le h_{\sigma^*}$ everywhere on $[-\pi, \pi]$.
Therefore, applying Lemma~\ref{lemma-subh}, we conclude
that, almost surely, $F_\xi$ has completely regular growth on the ray
$\{ \arg (z) = \theta \}$ with the indicator $h_{\sigma^*}(\theta)$. To complete the proof, we apply this argument
to a dense countable set of $\theta$'s. \hfill $\Box$

\section{Proof of Theorem~5}

Now, $\xi$ is a uniformly almost-periodic sequence.
By $\widehat{\xi}$ we denote the Fourier transform of $\xi$, $\widehat{\xi}\colon \bT \to \bC$.
The spectrum of $\xi$ is $\sigma(\xi) =\{e^{\im\la}\in\bT\colon \widehat{\xi}(e^{\im\la})\ne 0\}$,
this is an at most countable subset of $\bT$.

We will be using Bochner's theorem that states that
{\em there exists an enumeration of the spectrum
$\sigma (\xi) = \{e^{\im\la_1}, e^{\im\la_2}, \ldots\, \}$ and a sequence of multipliers $\beta_k^{(m)}$}
($k\in\{1, \ldots, m\}$) {\em satisfying  $0\le \beta_k^{(m)}\le 1$ and $\beta_k^{(m)}\to 1$ as
$m\to\infty$, $k$ stays fixed, such that the finite exponential sums
\[
\sum_{k=1}^m \beta_k^{(m)} \widehat{\xi}(e^{\im\la_k}) e^{\im\la_k n}
\]
converge to $\xi (n)$ uniformly in $n\in\bZ$ as $m\to\infty$}.
For the proof, see, for instance, \cite[Chapter~VI, \S~5]{Katznelson}. Therein, the proof is given for almost
periodic functions, the proof for almost periodic sequences is almost the
same.

As before,
by $\sigma^*$ we denote the reflection of $\sigma(\xi)$ in the real axis.
First, we show that $h^F\le h_{\sigma^*}$ everywhere, and then that $|F(re^{\im\theta})|\ge c(\theta)e^r$
with some $c(\theta)>0$, whenever $\theta\in\sigma^*$ and $r\ge r_0(\theta)$. Then Lemma~\ref{lemma2} does the job.

\subsection{}

The following lemma is an old result of Bochner and Bonnenblust~\cite{BB}.
The proof given here follows the one in~\cite[Chapter~VI]{Levin}.

\begin{lemma}\label{lemma-ap1} Everywhere,
$h^{F_\xi}\le h_{\sigma^*}$.
\end{lemma}

\begin{proof}

If the spectrum $\sigma (\xi)$ is dense on $\bT$, then $h_{\sigma^*}\equiv 1$, and
there is nothing to prove. So we assume that there is an open arc $J\subset\bT$ such
that $\sigma(\xi) \bigcap J = \emptyset$. Rotating the complex plane, $z\mapsto ze^{-\im t}$,
we shift the spectrum $\sigma^*(\xi)$ and the indicator function $h^{F_\xi}$ by $t$. Therefore, without loss of generality,
we may assume that $\sigma(\xi)$ is contained in the arc $\{e^{\im\theta}\colon |\theta|\le \pi-\delta \}$ for some $\delta>0$.
We need to show that the indicator diagram $I^F$ is contained in the closed convex hull
of $\{e^{\im\theta}\colon |\theta|\le \pi-\delta \}$.

By our assumption, the functions
\[
w \mapsto \Xi_m(w) = \sum_{k=1}^m \beta_k^{(m)} \widehat{\xi}(e^{\im\la_k}) e^{\im\la_k w}
\]
are entire functions of exponential type at most $\pi-\delta$. By Bochner's theorem, given $\e>0$, there exists
$M_\e$ so that, for all $m_1, m_2>M_\e$,
\[
\| \Xi_{m_1} - \Xi_{m_2} \|_{\ell^\infty(\bZ)} < \e.
\]
Then, by Cartwright's theorem~\cite[Chapter~IV, Theorem~15]{Levin},
\[
\| \Xi_{m_1} - \Xi_{m_2} \|_{L^\infty(\bR)} < C(\delta)\e,
\]
and, invoking one of the Phragm\'en--Lindel\"of theorems, we conclude that the
sequence of entire functions $\Xi_m$ converges to an entire function $\Xi$ uniformly in
any horizontal strip. Obviously, the entire function $\Xi$ interpolates the sequence $\xi$ at $\bZ$,
the exponential type of $\Xi$ does not exceed $\pi-\delta$, and $\Xi$ is bounded on $\bR$.
Thus, the indicator diagram of $\Xi$ is contained in the interval $[(-\pi+\delta)\im, (\pi-\delta)\im]$
of the imaginary axis. It is worth noting that in what follows we use only that the exponential type of $\Xi$ does not exceed $\pi-\delta$.

Now, consider the Taylor series
\[
f(s) = \sum_{n\ge 0} \xi (n)s^n
\]
analytic in the unit disk.
Since the coefficients $\xi(n)$ can be interpolated by an entire function of exponential type
at most $\pi-\delta$, the function $f$ can be analytically continued through the arc
$\{e^{\im\theta}\colon |\theta-\pi|<\delta \}$ to $\bar\bC\setminus\bar\bD$. This is a classical
result that goes back to Carlson and P\'olya (see~\cite[Appendix~1, \S~5]{Levin}
or~\cite[Theorem~1.3.1]{Bieberbach}).
On the other hand, the function $w^{-1} f(w^{-1})$ is nothing but the Laplace transform of the
entire function $F_\xi$, and, as we have seen,
this function is analytic outside the closed convex hull of the arc $\{e^{\im\theta}\colon |\theta|\le \pi-\delta \}$.
Then, by P\'olya's theorem (see~\cite[Chapter~I, Theorem~33]{Levin} or~\cite[Theorem~1.1.5]{Bieberbach}),
the indicator diagram $I^{F_\xi}$ is contained in the closed convex hull of $\{e^{\im\theta}\colon |\theta|\le \pi-\delta \}$.
This completes the proof of the lemma.
\end{proof}

\subsection{}
Here, we show that $F_\xi$ grows as $e^r$ on the rays corresponding to the set $\sigma^*$.

\begin{lemma}\label{lemma-ap2}
For every $\theta\in\sigma^*$, there exists $c(\theta)>0$ and $r(\theta)<\infty$ so that
\[
\bigl| F_\xi (re^{\im\theta}) \bigr| \ge c(\theta) e^r, \qquad r\ge r(\theta)\,.
\]
\end{lemma}

\begin{proof} Once again, we will be using Bochner's theorem.
We fix $e^{\im\la_j}\in\sigma(\xi)$, take $m\ge j$, and put
\[
\xi_m (n) = \sum_{k=1}^m \beta_k^{(m)}\, \widehat{\xi}(e^{\im\la_k}) e^{\im\la_k n}\,.
\]
Then, uniformly in $z$,
\begin{equation}\label{eq:A4}
\bigl| F_\xi (z) - F_{\xi_m} (z) \bigr| \le \e_m e^{|z|},
\qquad {\rm with}\ \e_m\to 0.
\end{equation}
Furthermore, $F_{\xi_m}(z)$ is a finite sum of exponential functions
\[
F_{\xi_m}(z) = \sum_{k=1}^m \beta_k^{(m)} \widehat{\xi}(e^{\im\la_k}) e^{z e^{\im\la_k}},
\]
whence
\begin{align*}
\bigl| F_{\xi_m}(re^{-\im\la_j}) \bigr| &\ge \beta_j^{(m)} |\widehat{\xi}(e^{\im\la_j})| e^r
- \sum_{\substack{k=1\\ k\ne j}}^m \beta_k^{(m)} |\widehat{\xi}(e^{\im\la_k})|
e^{r\cos (\la_k-\la_j)} \\
&\ge \beta_j^{(m)} |\widehat{\xi}(e^{\im\la_j})| e^r  - C_m e^{(1-\delta_m)r}\,,
\end{align*}
with some $\delta_m>0$. Therefore,
\begin{equation}\label{eq:A5}
\liminf_{r\to\infty} e^{-r}\, \bigl| F_{\xi_m}(re^{-\im\la_j}) \bigr|
\ge \beta_j^{(m)} |\widehat{\xi}(e^{\im\la_j})| \ge \tfrac12  |\widehat{\xi}(e^{\im\la_j})|\,,
\end{equation}
provided that $m\ge m_0(j)$.  Juxtaposing~\eqref{eq:A4} and~\eqref{eq:A5},
we get Lem\-ma~\ref{lemma-ap2}.
\end{proof}

\medskip To finish off the proof of Theorem~\ref{thm:Levin}, we
observe that, by Lemmas~\ref{lemma-ap1} and~\ref{lemma-ap2},
the function $F_\xi$ satisfies the assumptions of Lemma~\ref{lemma2},
and then Theorem~\ref{thm:Levin} readily follows. 
\hfill $\Box$

\renewcommand{\thesection}{}
\sectionmark{References}
\renewcommand{\sectionmark}[1]{\markright{\thesection.\ #1}}

\bigskip
\medskip

{Alexander Borichev:
Aix Marseille Universit\'e, CNRS, Centrale Marseille, I2M, UMR 7373, 13453 Marseille,
France
\newline {\tt borichev@cmi.univ-mrs.fr}
\smallskip
\newline \phantom{x}\,\, Alon Nishry:
 Department of Mathematics, University of Michigan, Ann Arbor, USA
\newline {\tt alonish@umich.edu}
\smallskip
\newline \phantom{x}\,\, Mikhail Sodin:
 School of Mathematics, Tel Aviv University, Tel Aviv, Israel
\newline {\tt sodin@post.tau.ac.il}
}

\end{document}